\documentclass[12pt]{amsart}
\usepackage{amsmath, amssymb}
\usepackage{amsfonts}
\usepackage{amsthm}
\usepackage{amscd}
\usepackage{enumerate}
\usepackage{mathtools}
\usepackage[pdftex,colorlinks,linkcolor=magenta,citecolor=cyan]{hyperref}
\usepackage[dvipsnames]{xcolor}
\usepackage[margin=2.5cm]{geometry}

\usepackage{mathrsfs}
\usepackage{euscript}
\usepackage{bbm}
\usepackage{upgreek}
\usepackage{bm}
\allowdisplaybreaks

\hypersetup{
	linktocpage,
	colorlinks,
	linkcolor=blue,   
	citecolor=red,    
	urlcolor=teal     
}

\theoremstyle{plain}
\newtheorem{thm}{Theorem}[section]
\newtheorem{lemma}[thm]{Lemma}
\newtheorem{prop}[thm]{Proposition}

\newtheorem{cor}[thm]{Corollary}
\theoremstyle{definition}
\newtheorem{dfn}[thm]{Definition}
\newtheorem{rem}[thm]{Remark}

\newtheorem*{thm*}{Theorem}

\numberwithin{equation}{section}

\makeatletter
\renewcommand{\subsubsection}{\@startsection{subsubsection}{3}
	{\z@}{1.3ex \@plus .8ex \@minus .2ex}{-0.5em}%
	{\normalfont\bfseries}}
\makeatother

\DeclareMathAlphabet{\mathesstixfrak}{U}{esstixfrak}{m}{n}

\DeclareFontFamily{U}{Tempora-TLF}{}
\DeclareFontShape{U}{Tempora-TLF}{m}{it}{<-> Tempora-Italic-TLF-lgr}{}

\newcommand{\bR}{\mathbb{R}}

\newcommand{\SO}{\mathsf{SO}}

\newcommand{\sL}{\mathsf{L}}
\newcommand{\tL}{\mathtt{L}}
\newcommand{\GL}{\mathsf{GL}}

\newcommand{\M}{\mathrm{M}}

\newcommand{\U}{\mathsf{U}}

\newcommand{\bP}{\mathbb{P}}
\newcommand{\tM}{\mathtt{M}}
\newcommand{\rT}{\mathrm{T}}
\newcommand{\sd}{\mathsf{d}}
\newcommand{\tR}{\mathtt{R}}

\newcommand{\tH}{\mathtt{H}}
\newcommand{\Z}{\mathtt{Z}}

\newcommand{\rL}{\mathrm{L}}

\newcommand{\sP}{\mathsf{P}}

\newcommand{\sG}{\mathsf{G}}
\newcommand{\PSL}{\mathsf{PSL}}

\newcommand{\fk}{\mathesstixfrak}

\newcommand{\Hom}{\mathrm{Hom}}

\newcommand{\Th}{\mathrm{Th}}
\newcommand{\HR}{\mathrm{HR}}
\newcommand{\Top}{\mathrm{top}}

\newcommand{\BM}{\mathrm{BM}}

\title{Thurston's asymmetric metric on Margulis spacetimes }

\author{Krishnendu Gongopadhyay} 

\address{Indian Institute of Science Education and Research (IISER) Mohali, Knowledge City,  Sector 81, S.A.S. Nagar 140306, Punjab, India}
\email{krishnendu@iisermohali.ac.in}

\author{Neelanjan Mondal}

\address{Indian Institute of Science Education and Research (IISER) Mohali, Knowledge City,  Sector 81, S.A.S. Nagar 140306, Punjab, India}
\email{mondalneelanjan@gmail.com}
\subjclass[2020]{Primary 22E40; Secondary 53C60, 37D35, 20F67}

\keywords{domain of discontinuity, Anosov representaion, pullback bundles, space of connections, categorical structure}

\date{\today}

\thanks{}

\begin{document}
		
		\begin{abstract}

        In this article, we extend Thurston's asymmetric metric and the associated Finsler norm, originally defined for Teichmüller space, to the setting of Margulis spacetimes. We also establish several convexity properties of both the asymmetric metric and the corresponding Finsler norm.

		\end{abstract}
		
		\maketitle

		\section{Introduction}

        Teichmüller theory provides one of the classical settings in which geometric structures on a manifold can be parametrized by a space of group representations.  
For a closed surface $S$, the Teichmüller space $\mathrm{Teich}(S)$ identifies with a connected component of the character variety
$$
\mathcal{X}(\pi_1(S),\PSL_2(\mathbb R))
    = \Hom(\pi_1(S),\PSL_2(\mathbb R))/\PSL_2(\mathbb R),
$$
and every point of $\mathrm{Teich}(S)$ corresponds to a Fuchsian representation encoding the hyperbolic geometry on $S$.

For a nontrivial element $\gamma\in\pi_1(S)$, the translation length
$$
l_\rho(\gamma)
    =\inf_{x\in\mathbb H^2} d_{\mathbb H^2}(x,\rho(\gamma)x)
$$
records the length of the closed geodesic on the hyperbolic surface $\mathbb H^2/\rho(\pi_1(S))$ represented by $\gamma$.  Thurston's asymmetric metric, introduced in his study of minimal stretch maps between hyperbolic surfaces~\cite{thurston1998minimal}, has become one of the most influential geometric structures on Teichm\"{u}ller space. Thurston wanted a metric that reflects the geometry of how one hyperbolic structure can be deformed into another by stretching along a measured lamination. Thurston showed that comparing the length spectrum for two hyperbolic structures yields an asymmetric metric on $\mathrm{Teich}(S)$,
 \[
    d_\Th(\rho, \varrho):= \sup_{[\gamma] \in [\pi_1(S)]} \log \left(\frac{l_{\rho}(\gamma)}{l_{\varrho}(\gamma)}\right).
        \] 
We refer to the surveys \cite{pasu}, \cite{pasu2} for more details on Thurston's metric on surfaces. 

    More recently, the Thurston's asymmetric metric construction has been generalized in higher Teichmüller theory. Carvajales, Dai, Pozzetti and Wienhard \cite{carvajales2024thurston} have  extended Thurston’s construction to Anosov representations of word hyperbolic groups into semisimple Lie groups.  Therefore, a natural question is whether one can develop an analogue of the Thurston asymmetric metric in the \emph{non-semisimple} setting.  
In this article we answer this question for an important class of geometric structures arising from the  {Margulis spacetimes}.

Margulis~\cite{margulis1983free, margulis1987complete} constructed free, non-abelian 
subgroups of $\GL_3(\mathbb{R}) \ltimes \mathbb{R}^3$ acting properly discontinuously 
and freely on $\mathbb{R}^3$, yielding the first counterexamples to the Auslander 
conjecture. Fried and Goldman~\cite{fried1983three} subsequently showed that the 
Zariski closure of any such group is conjugate to $\SO(2,1)$. Later, Abels--Margulis--Soifer~\cite{abels2002zariski} constructed Zariski-dense free 
subgroups of $\SO(n+1,n)\ltimes \mathbb{R}^{2n+1}$ that act properly on 
$\mathbb{R}^{2n+1}$ precisely when $n$ is odd. The geometric manifolds obtained in 
this way are commonly referred to as \emph{Margulis spacetimes}.

Recall that $\SO(n+1, n)$ is the group of isometries of the real non-degenerate qadratic form $\langle \cdot ~|~\cdot \rangle$ of signature $(n+1, n)$.  We consider Margulis spacetimes modeled on
$$
\sG:=\SO^0(n+1,n)\ltimes\mathbb R^{2n+1}.
$$
Let \(\sG^\tL := \SO^0(n+1,n)\), the identity component of $G$. 

Let $\Gamma$ is word hyperbolic group, and let  $\rho : \Gamma \rightarrow \sG$ be an injective homomorphism such that the linear part $\tL_{\rho} \in \Hom(\Gamma, \sG^\tL, \sP^\pm)$. Then $\rho$ is called a \textit{Margulis spacetime} if and only if $\rho(\Gamma)$ acts properly on $\bR^{2n+1}$. The \emph{linear part} of $\rho$: $\tL_\rho: \Gamma \to \SO^0(n+1,n)$ 
    describes how elements of $\Gamma$ act by linear transformations on $\mathbb{R}^{2n+1}$. 
    We assume that $\tL_\rho$ is \emph{Zariski dense} and \emph{Anosov}. We denote the space of {Margulis spacetimes} by $\Hom_\M(\Gamma, \sG)$.

Associated to such a representation $\rho$,   the Margulis invariant $\alpha_{\rho}$  which measures the signed translation along the affine axis of each group element. Margulis invariant detects proper affine action, for example, see \cite{margulis1983free}, \cite{gold}, \cite{goldman2009proper}. 
    
We say an element $g \in \SO(n+1,n)$ is \textit{pseudo-hyperbolic} if and only if $-1$ is not an eigenvalue of $g$ and it has an eigenvalue $1$ with geometric multiplicity one.   An  element $(g, u) \in \SO(n+1,n)\ltimes \bR^{2n+1}$ is called \textit{pseudo-hyperbolic} if the  linear part $g$ is pseudo-hyperbolic.  Now for a pseudo-hyperbolic element $(g,u) \in \SO(n+1,n) \ltimes \bR^{2n+1}$, the stable and unstable subspaces are defined by   
$$V^g_{\pm} :=\{v \in \bR^{2n+1} 
 \mid \lim_{k \to \infty} g^{\mp k} v=0\}.$$ These are maximal isotropic subspaces of $\bR^{n+1,n}$. Let $v_0^g$ be  the  unique eigenvector of $g$ with eigenvalue $1$ such that $\langle v_0^g \mid v_0^g \rangle = 1$ and is positively oriented with respect to the positive orientation of $V_+^g$ and $(V_+^g)^\perp \cap \bR^{n+1,n}$. The \textit{Margulis invaiant} of $g$ is defined by $$\alpha(g,u) := \langle u \mid v_0^g \rangle.$$ 

The above definition for Margulis invariant has been reformulated for representations $\rho: \Gamma \to G$. We recall the definition from  \cite{ghosh2023affine}.  First recall the notion of \emph{neutral section.}

If $\rho:\Gamma\to G$ has Anosov linear part $L_\rho$ with limit map
$\xi_{L_\rho}$ and $\nu$ denotes the neutral map, the \emph{neutral section} is
\[
\nu_\rho := \nu \circ \xi_{L_\rho} : \partial_\infty\Gamma^{(2)} \to \mathbb{R}^{2n+1}.
\]

  \begin{dfn}  Suppose $\rho : \Gamma \to \sG$ be an injective homomorphism such that $\tL_{\rho} \in \Hom(\Gamma, \sG^\tL, \sP^\pm)$. Then for any infinite order element $\gamma \in \Gamma$, the \textit{Margulis invariant} is given by 
    \[
    \alpha_{\rho} (\gamma) = \langle u_{\rho}(\gamma) \mid \nu_{\rho}(\gamma^+, \gamma^-)\rangle,
    \]
    where $\gamma^-, \gamma^+$ are the repelling and attracting fixed points of $\gamma$, respectively, and,  $\nu_{\rho} (\gamma^+, \gamma^-): = \nu \circ \xi(\gamma^+, \gamma^-,t)$, with $\nu$ is the neutral map. 
\end{dfn}

The moduli space of such representations, up to conjugation in $\sG$, is
$$
\M(\Gamma, \sG) := \Hom_\M(\Gamma, \sG) / \sG.
$$
For any positive $k$, we define the analytic hypersurface
$$
\M_k(\Gamma, \sG) := \M(\Gamma, \sG) \cap h^{-1}(k),
$$
where $h: \M(\Gamma ,\sG) \to \mathbb{R}_{>0}$ is the \emph{entropy functional}, which measures the dynamical complexity of the representation. By \cite{ghosh2018pressure}, the space $\M(\Gamma,\sG)$ is an analytic manifold. 
In addition, \cite{ghosh2023margulis} shows that the entropy map
\[
\rho \longmapsto h_\rho
\]
is differentiable, and likewise the  function
\[
\rho \longmapsto \alpha_\rho(\gamma)
\]
is differentiable for each $\gamma \in \Gamma$. Thus, both the global dynamical invariant $h_\rho$ and the function $\alpha_{\rho}$  vary smoothly along $\M(\Gamma,\sG)$.

In this article, our aim is to offer an analogue of Thurston's asymmetric metric on $\M(\Gamma, \sG)$. 
Since $\SO^0(n+1,n)\ltimes\mathbb R^{2n+1}$ is non-semisimple, the results of~\cite{carvajales2024thurston} do not apply directly; nevertheless, the underlying strategy inspires our construction.  
Our first main result establishes a Thurston-type asymmetric metric on this moduli space.

        \begin{thm}
            The function $d_\Th: \M(\Gamma, \sG) \times \M(\Gamma, \sG) \rightarrow \bR \cup \{\infty\}$ defined by  $$d_{\Th}(\rho, \varrho) := \log\left(\sup_{[\gamma] \in [\Gamma_\tH]} \frac{h_{\rho}}{h_{\varrho}} \frac{\alpha_{\rho}(\gamma)}{\alpha_{\varrho}(\gamma)}\right).$$
            is an  asymmetric metric on $\M(\Gamma, \sG)$.
        \end{thm}
    Note that the renormalization by entropy in the definition of asymmetric metric is essential. For instance, let $\rho = (\tL_{\rho}, u_{\rho})$ and $\varrho = (\tL_{\varrho}, u_{\varrho})$ are two Margulis spacetimes with $\tL_{\rho} = \tL_{\varrho}$ and $u_{\varrho} = cu_{\rho}$ for some $c>1$. Then $\alpha_{\varrho}(\gamma) =c \alpha_{\rho}(\gamma)$ for all $\gamma \in \Gamma$. Hence  
    $$\log \left(\sup_{[\gamma] \in [\Gamma_\tH]} \frac{\alpha_{\rho}(\gamma)}{\alpha_{\varrho}(\gamma)}\right) = \log  \left(\frac{1}{c}\right) <0.$$

    On the other hand, for $\rho, \varrho \in \M(\Gamma, \sG)_k$ (See Section~3.2 of \cite{ghosh2023margulis} for Margulis  spacetimes with constant entropy), the entropy $h_{\rho}$ and $h_{\varrho}$ are fixed. In these situations, renormalizing by entropy is not needed, and the Thurston's asymmetric metric simplifies to 
    $$d_\Th(\rho, \varrho) = \log \left(\sup_{[\gamma] \in [\Gamma_\tH]} \frac{\alpha_{\rho}(\gamma)}{\alpha_{\varrho}(\gamma)}\right).$$

        This metric leads naturally to a renormalized marked-spectrum rigidity statement. The general rigidity result follow from a theorem by Kim \cite[Theorem 3]{kim2005affine}. We give the proof for completeness. However, we also outline an independent  proof fixing the linear part. This may be viewed as a renormalized analogue of the rigidity theorem of Kim~\cite{kim2005affine}. 
\begin{thm}\label{Th:Rigidity}
Let $\rho,\varrho:\Gamma\to\sG$ be Margulis spacetimes.  
If $h_\rho \alpha_\rho = h_{\varrho} \alpha_{\varrho}$, then there exists an isomorphism
$\sigma:\sG \to \sG$
such that $\sigma\circ\rho=\varrho$.  
In particular, if $\rho,\varrho \in \M_k(\Gamma, \sG)$, then $\rho$ and $\varrho$ are conjugate.
\end{thm}

 We then turn to the infinitesimal geometry of the moduli space.  
Differentiating the asymmetric metric yields a Finsler norm
$$
\|v\|_{\Th}
   = \sup_{[\gamma]\in[\Gamma_{\mathrm H}]}
     \frac{
       \dot{h}_{\rho}(v)\,\alpha_\rho(\gamma)
       + h_\rho\, \dot{\alpha}_{\rho}(\gamma))(v)
     }{h_\rho\,\alpha_\rho(\gamma)}
$$
defined on $T_\rho \M(\Gamma,\sG)$.  We refer to this norm as \emph{Thurston norm}.

\begin{thm}\label{Lemma:Tgnt Vec}
    Let $\{\rho_z\}_{z \in D} \subset \M(\Gamma, \sG)$ be a real analytic family of Margulis spacetimes parameterizing an open neighborhood of $\rho = \rho_0$. Fix an analytic  path $z: (-1,1) \rightarrow D$ with $z(0) = 0$, and let $\rho_s = \rho_{z(s)}$ and $v = \frac{d}{ds}\big|_{s=0} \rho_s \rho^{-1}$ be the tangent vector at $\rho$. Let also $h_s := h_{\rho_s}$, $\alpha_s := \alpha_{\rho_s}$ and $g_s:= h_s\widetilde{g}_{z(s)}$. Then 
    $$\lVert v\rVert_\Th = \lVert[\dot{g_0}]\rVert_\Th,$$
    where $\dot{g}_0 = \frac{d}{ds}\big|_{s=0}  g_s$ and the  Liv\v{s}ic cohomology class $\left[\dot{g_0}\right]$  belongs to the tangent space $T_{[\phi^{g_0}]} \bP\HR^\upsilon (\phi)$. 
\end{thm}

\begin{thm}\label{Theorem:Properties of Finsler norm}
 Let $\rho \in \M(\Gamma, \sG)$ be a Margulis spacetime admitting an analytic neighborhood in $\M(\Gamma, \sG)$. Then  $\lVert \cdot\rVert_\Th : \rT_{\rho}\M(\Gamma, \sG) \to \bR \cup \{\pm \infty\}$ is real valued and non-negative. Furthermore, it is $\bR_{>0}$-homogeneous, satisfies the triangle inequality and 
    \[
    \lVert v\rVert_\Th = 0 \Leftrightarrow \dot{\alpha}_{\rho} (\gamma))(v)= -\frac{\dot{h}_{\rho}(v)}{h_{\rho}}\alpha_{\rho}(\gamma)
    \]
    for all $\gamma \in \Gamma_\tH$. In particular, if the function $\rho \mapsto h_{\rho}$ is constant, that is, if $\rho\in \M_k(\Gamma, \sG)$, then 
    \[
    \lVert v\rVert_\Th = 0 \Leftrightarrow \dot{\alpha}_{\rho}(\gamma)(v) = 0
    \]
    for all $\gamma \in \Gamma_\tH$. In this case, the function $\lVert \cdot\rVert_\Th : \rT_{\rho}\M(\Gamma, \sG) \to \bR \cup \{+ \infty\}$
$$
\|v\|_{\Th}
 = \sup_{[\gamma]\in[\Gamma_{\mathrm H}]}
   \frac{d}{ds}\Big|_{s=0}\log\alpha_{\rho_s}(\gamma)
$$
defines a Finsler norm.
\end{thm}

Recall that the pressure metric introduced in \cite{bridgeman2015pressure}, \cite{bridgeman2018introduction} generalizes the Weil-Petersson metric from Teichmüller space to Anosov representations, providing a higher-rank analogue of the classical Teichmüller geometry. The pressure metric uses thermodynamical formalism to define an intersection function for projective Anosov representations. This framework establishes analyticity of both the intersection and entropy, and yields rigidity results for the intersection, offering a powerful tool for studying the geometry of higher-rank representation spaces. The pressure metric was extended to the setting of Margulis spacetimes in \cite{ghosh2023margulis}. 

We compare the Thurston norm with the pressure metric and establish convexity results for deformations with fixed Anosov linear part.  An additional analytic feature of our construction is that the Thurston Finsler norm is dominated by the pressure metric $\mathcal P$ (see Subsection~\ref{Pressure metric} for the definition).

\begin{prop}\label{domination}
For every Margulis spacetimes $\rho \in \M_k(\Gamma, \sG)$ and every 
$v\in T_\rho \M_k(\Gamma,\sG)$, the Finsler norm satisfies
$$
\|v\|_{\Th} \;\le\; C\,\sqrt{\mathcal{P}_\rho(v,v)},
$$
where $C>0$ is a constant depending only on the normalization of the Labourie-Margulis invariant $f_\rho$.  
\end{prop}

This inequality situates the asymmetric Thurston geometry of Margulis spacetimes within 
the analytic framework of the pressure metric, and shows that our Finsler norm is locally 
controlled by a smooth, symmetric metric.  Such domination phenomena mirror the 
relationship between Thurston’s metric and the Weil--Petersson metrics in 
classical Teichm\"{u}ller theory.

In addition to contructing the asymmetric metric and its infinitesimal Finsler structure, we also established several new convexity phenomena for Margulis spacetimes. 

\begin{thm}\label{Theorem:Convexity intro}
    Let $\Gamma$ be a word hyperbolic group and let $\rho_0, \rho_t \in \M_k(\Gamma, \sG)$ be two Margulis spacetimes with constant entropy $k$.  Let $\rho_t = (\tL, (1-t)u_0 + tu_1)$ for $t \in [0,1]$, where $\tL$ is a fixed Anosov linear part and $u_0,u_1$  are translational parts such that $\rho_t \in \M_k(\Gamma, \sG)$ for all $t$. Then

    $(1)$ The function $t \mapsto d_\Th(\rho, \rho_t)$ is convex on $[0,1]$.

    $(2)$ The function  
 $t \mapsto -\log \alpha_{\rho_t}$ is convex on $[0,1]$.

 $(3)$ The function $t \mapsto -\lVert v\rVert_\Th$ is convex on $[0,1]$.
\end{thm}

\subsection*{Structure of the paper}

Section~\ref{Section 2} recalls the necessary background materials. In Section~\ref{Section 3}, we construct Thurston’s asymmetric metric on $\M(\Gamma, \sG)$ (resp. $\M_k(\Gamma, \sG)$), and prove the renormalized marked Margulis invariant rigidity theorem, i.e.  Theorem~\ref{Th:Rigidity}. In Section~\ref{Section 4}, we define the  Finsler norm on the tangent space of $\M(\Gamma, \sG)$ (resp. $\M_k(\Gamma, \sG)$), and compare it to the pressure metric. In Section~ \ref{Section:5}, we establish convexity results for deformations with fixed Anosov linear part. We prove  Theorem~\ref{Theorem:Convexity intro} in this section.

\section{Preliminaries} \label{Section 2}
        \subsection{Thermodynamical formalism}     In this subsection, we briefly recall fundamental concepts in the thermodynamic formalism dynamics of topological flows on compact metric spaces. It was developed by Bowen, Parry–Pollicott, Pollicott, Ruelle, for example, see  \cite{BowenRuelle}, \cite{bowen0}, \cite{pollicott1987symbolic}, \cite{parry1990zeta}, \cite{ruelle2004thermodynamic}. In higher Teichm\"{u}ller theory, these were introduced by Bridgeman-Canary-Labourie-Sambarino (See \cite{bridgeman2015pressure}).

       In the following, we briefly recall basic notions and results that are instrumental for our work, closely following the expositions of
  \cite[Section 3]{bridgeman2015pressure} and \cite[Section 2]{carvajales2024thurston}.

        \begin{dfn}
    Let $\phi$ be a H\"older continuous flow on a compact metric space $X$, and let $f: X \rightarrow \bR$ be a positive H\"older continuous function. Define $\alpha_f(x,t)$ for $(x,t) \in X \times \bR$ by 
    $$\int_{0}^{\alpha_f(x,t)} f(\phi_s(x))ds = t.$$
    Then the flow $\phi^f_t(x):= \phi_{\alpha_f(x,t)} (x)$ is called the \textit{H\"older reparametrization} of $\phi$ by $f$, and the  space of all such reparametrization of $\phi$ is denoted by $\HR(\phi)$.
\end{dfn}

\begin{dfn}\label{dfn: livsic}
Let $\phi$ be a H\"older continuous flow on a compact metric space $X$. Two H\"older continuous functions $f,g:X\to \bR$ are  \textit{Liv\v{s}ic cohomologous} (denoted by $f \sim_{\phi} g$), if  there exists  a H\"older continuous function $h:X\to\bR$, differentiable  along the flow direction,  such that   
$$f-g=\left.\frac{d}{d t}\right\vert_{t=0} h \circ\phi_t.$$
We write $[f]_{\phi}$ for the Liv\v{s}ic cohomology class of $f$.
\end{dfn}

\begin{dfn}\label{Def: Variation}
Let $\phi$ be a continuous flow on a compact metric space $X$, and let $\mathscr{P}({\phi})$ be the $\phi$-invariant probability measures on $X$. For a continuous function $f: X\to \bR$, the \textit{topological pressure} of $f$ with respect to $\phi$ is  
\begin{equation*}\label{eq: def pressure}
    \textbf{P}(\phi,f):=\sup\limits_{\mu\in\mathcal{P}({\phi})}\left(h(\phi,\mu)+ \int f  d\mu\right),
\end{equation*}
\noindent where $h(\phi,\mu)$ is the \textit{metric entropy} of $\mu$ with respect to the flow $\phi$. The pressure of the zero function is called  the \textit{topological entropy} of $\phi$, denoted by $h_{\phi}$. A  measure $\mu \in \mathscr{P}(\phi)$ realizing the supremum  is called an \textit{equilibrium state} of $f$.
\end{dfn}

\begin{prop}[\cite{BowenRuelle}]\label{prop: pressureZero}
    Let $\phi$ be a H\"older continuous flow on a compact metric space $X$ and $f$ is a positive  H{\"o}lder continuous function on $X$. Then 
    $$\textnormal{\textbf{P}}(\phi, -hf)=0 \Leftrightarrow h=h_{\phi^{f}}.$$ 
\end{prop}

 Let $O$ be the set of periodic orbits of the flow $\phi$, and  $p_{\phi}(a)$ denotes the period of  $a \in O$.

\begin{prop}[\cite{bowen1972periodic}, \cite{pollicott1987symbolic}]
Let $\phi$ be a topologically transitive metric Anosov flow on a compact metric space $X$ and $f: X \to \bR$ is a positive H\"older continuous function. Then 
    \begin{equation*}\label{eq: entropy}
    h_{\phi^f} = \displaystyle\lim_{t\to\infty}\frac{1}{t}\log\#\{a\in O\mid\int_{0}^{p_{\phi}(a)} f(\phi_t(x)) dt\leq t\},
\end{equation*}
\noindent and this limit is positive and finite.
\end{prop}

\begin{thm}[\cite{BowenRuelle},  \cite{pollicott1987symbolic},  \cite{parry1990zeta}]
    Let $\phi$ is a be a topologically transitive metric Anosov flow on a compact metric space $X$, and let  $f$ be a positive H\"older continuous on $X$. Then there exists a unique  equilibrium state $\mu_{f}$ for $f$ with respect to $\phi$.
\end{thm}

The equilibrium state for zero function is the \textit{Bowen-Margulis measure} of $\phi$, denoted by $\mu^\BM:=\mu_{0}.$

\begin{thm}[Liv\v{s}ic {\cite{Livsic}}] \label{thm:Livsic} Let $\phi$ be a topologically transitive metric Anosov flow on a compact metric space $X$. Then two H\'older continuous functions  $f,g :X \to \bR$ are Liv\v{s}ic cohomologous if and only if 
$$\int_0^{p_{\phi}(a)} f(\phi_t(x))dt = \int_0^{p_{\phi}(a)} g(\phi_t(x))dt,$$
for all $a \in O$ and  $x \in a$. 
\end{thm}

\begin{thm}[{\cite{parry1990zeta}}]\label{Analyticity of pressure}
    Let $\phi$ be a topologically transitive metric Anosov flow on a compact metric space $X$, and let $f,g:X\to\bR$ be H\"older continuous. Then the pressure map  $$s\mapsto\mathbf{P}(\phi,f+sg)$$ \noindent is analytic. Moreover, the derivative at $s=0$ is given by    
        $$\frac{d}{ds}\bigg\vert_{s=0} \textnormal{\bf P}(\phi, f+sg)= \int g\ d \mu_{f},$$ where $\mu_f $ is the equilibrium state of $f$ with respect to the flow $\phi$.
\end{thm}

We now assume that  $\psi := \phi^f$ and $\widehat{\psi} := \phi^g$ are two H\"older reparametrization of $\phi$, where   where $f,g: X \to \bR$ are two positive H\"older continuous functions, i.e., $\psi, \widehat{\psi} \in \HR(\phi)$.

\begin{thm}[{\cite{carvajales2024thurston}}]
    Let $\phi$ be a topologically transitive metric Anosov flow on a compact metric space $X$, and let $\mu \in \mathcal{P}(\phi)$. The quantity 
    $$\textnormal{\textbf{I}}_{\mu}(\psi, \widehat{\psi}):= \int\frac{g}{f} d\mu,$$ 
    is called the $\mu$-\textit{intersection number} between $\psi$ and $\widehat{\psi}$ and satisfies 
    $$\sup_{a \in O} \frac{p_{\widehat{\psi}}(a)}{p_{\psi}(a)}= \sup_{\mu \in \mathcal{P}(\phi)} \textnormal{\textbf{I}}_{\mu}(\psi, \widehat{\psi}),$$ where $O$ denotes the set of periodic orbits of $\phi$. 
\end{thm}

\begin{lemma}[\cite{carvajales2024thurston}]\label{Lem:Ther-Equ}
    For $\psi, \widehat{\psi}\in \HR(\phi)$, the following are equivalent:

    $(1)$ For every $a \in O$, $h_{\widehat{\psi}} p_{\widehat{\psi}}(a) = h_{\psi} p_{\psi}(a)$.

    $(2)$ $(h_{\widehat{\psi}}g) \sim_{\phi} (h_{\psi}f)$
\end{lemma}

\subsection{H\"{o}lder reparametrizations}\label{hr}
We say that $\psi \sim \widehat{\psi}$ in $\HR(\phi)$ if $h_{\psi} f$ and $h_{\widehat{\psi}}g$ are Liv\v{s}ic cohomologous, and denote by $\bP\HR(\phi): = \HR(\phi)/\sim$, and $[\psi] \in \bP\HR(\phi)$ the equivalence class of $\psi$.

Let $\mathcal{H}^{\alpha}(X)$ denote the Banach space of $\upsilon$-H\"older continuous functions on $X$. By Liv\v{s}ic's Theorem \ref{thm:Livsic}, the subspace $\mathcal{B}^{\upsilon}(X) \subset \mathcal{H}^{\upsilon}(X)$, consisting $\upsilon$-H\"older continuous functions Liv\v{s}ic cohomologous to zero, is closed.  Hence the resulting quotient space $$\mathcal{L}^{\upsilon}(X): = \mathcal{H}^{\upsilon}(X)/\mathcal{B}^{\upsilon}(X)$$ is a Banach space. 

We define $\HR^{\upsilon}(\phi)$ be the set  of reparametrizations of $\phi^f$ of $\phi$ such that  $f \in \mathcal{H}^{\upsilon}(X)$, and denote its image in the projective space $\bP\HR(\phi)$ by $\bP\HR^{\upsilon}(\phi)$.

We also define the set of  Liv\v{s}ic cohomology class of pressure-zero H\"older continuous functions $$\mathcal{P}^{\upsilon}(X) :=\left\{[f]_{\phi}\in\mathcal{L}^{\alpha}(X): \mathbf{P}(\phi,-f)=0\right\}.$$

Then $\bP\HR^{\upsilon}(\phi)$ identifies   with the open subset of $\mathcal{P}^{\upsilon}(X)$ consisting of Liv\v{s}ic cohomology classes of pressure zero, positive, $\upsilon$-H\"older continuous functions on $X$.  For any $[f]_{\phi} \in \mathcal{P}^{\upsilon}(X)$,  the tangent space at $[f]_{\phi}$ is given by $$T_{[r]_{\phi}}\mathcal{P}^{\upsilon}(X)=\left\{[g]_{\phi}\in\mathcal{L}^{\upsilon}(X) \mid  \int g d\mu_{-f}=0 \right\},$$ where $\mu_{-f}$ is  the equilribium state of $-f$ with respect to $\phi$. The tangent space  of $\bP\HR^{\upsilon}(\phi)$ at a point  $[\psi]$, is given by $$T_{[\psi]}\bP\HR^{\upsilon}(\phi):=T_{[f]_{\phi}}\mathcal{P}^{\upsilon}(X).$$

The following metric was constructed in \cite{carvajales2024thurston}, and we will pull it back via a thermodynamical map to define an asymmetric metric in our context.

\begin{thm}[{\cite{carvajales2024thurston}}]\label{Th:Thasm-Therm}
Let $[\psi],[\widehat{\psi}] \in \bP\HR(\phi)$ be two reparametrization classes. Then the function 
    $$\sd_{\Th}([\psi], [\widehat{\psi}]) = \log \left(\sup_{a \in O} \frac{h_{\hat{\psi}}}{h_{\psi}}  \frac{p_{\widehat{\psi}}(a)}{p_{\psi} (a)}\right)$$
defines an asymmetric distance on $\bP\HR(\phi)$.
\end{thm}

Let $\lVert\cdot\rVert_\sd$

\begin{thm}[{\cite{carvajales2024thurston}}]\label{Thm:Fins1}
 Let $[\psi] \in \bP\HR^{\upsilon}(\phi)$ and $[g]_{\phi} \in T_{[\psi]}\bP\HR^{\upsilon}(\phi)$. Let  
$$\Vert [g]_{\phi} \Vert_{\sd}:=\displaystyle\sup_{m\in\mathcal{P}(\phi)}\frac{\int g d\mu}{\int f d\mu}.$$
Then $\Vert [g]_{\phi} \Vert_{\sd} \ge 0$ and $\Vert [g]_{\phi} \Vert_{\sd} = 0$ if and only if $[g]=0$. Moreover, $\lVert\cdot\rVert_{\sd}$ is $\bR_{>0}$-homogeneous and satisfies triangle inequality. Therefore, $\lVert\cdot\rVert_{\sd}$ defines a Finsler norm on the tangent space $T_{[\psi]}\bP\HR^{\alpha}(\phi)$. 
\end{thm}

\subsection{Thermodynamics of Margulis spacetimes}\label{Def: Pull back}

Let $\rho \in \Hom_\M(\Gamma, \sG)$, then the linear part $\tL \in \Hom^\Z(\Gamma, \sG^\tL, \sP^\pm)$. Since $\Gamma$ be a word hyperbolic group admitting an Anosov representation, then by Theorem~\ref{Th:MetricAnosov}, $\U\Gamma$ admits a topologically transitive metric Anosov flow $\phi$. 

\begin{dfn}
    The \emph{topological entropy} of $\rho \in \Hom_\M(\Gamma, \sG)$ is defined by 
\[
h_{\rho}:=h_\Top (\phi^{f_{\rho}}) =  \lim_{T \to \infty} \frac{1}{T}\log\left|R_T(\rho)\right|
\]
where $R_T(\rho) = R_T(f_{\rho})= \{\gamma \in O \mid \alpha_{\rho}(\gamma) \leq T\}$.  By \cite[Lemma~3.1]{ghosh2023affine}, the quantity $h_{\rho}$ is a finite positive number.
\end{dfn}

We now consider the following map
\begin{align*}
    \tR: \Hom_\M(\Gamma, \sG) &\to \bP\HR(\phi)\\
    \rho &\mapsto [\phi^{h_{\rho}f_{\rho}}]
\end{align*} 
We note that by Remark~\ref{Rmk: Conjugacy invar Margulis}, $\int_{\gamma} f_{g\rho g^{-1}}  d\mu_{\gamma} = \frac{\alpha_{g\rho g^{-1}}(\gamma)}{\ell(\gamma)} = \frac{\alpha_{\rho}(\gamma)}{\ell(\gamma)} = \int_{\gamma}f_{\rho} d\mu_{\gamma}$ for all $\gamma \in O$. By the Liv\v{s}ic Theorem~\ref{thm:Livsic}, this implies that $f_{g\rho g^{-1}}$ is Liv\v{s}ic cohomologous to $f_{\rho}$. Consequently, Lemma~\ref{Lem:Ther-Equ}, yields $\tR(g\rho g^{-1}) = \tR(\rho)$. In particular, $\tR$ is constant on conjugacy classes. It follows that $\tR$ descends to the quotient, and we continue to denote the induced map by $\tR$
\begin{align*}
    \tR: \M(\Gamma, \sG) &\to \bP\HR(\phi)\\
    [\rho] &\mapsto [\phi^{h_{\rho}f_{\rho}}]
\end{align*} 
For notational convenience, we will abuse notation and write $\rho$ in place of its conjugacy class $[\rho]$ throughout the paper.

        \subsection{Anosov representations}
    Anosov representations are discrete, faithful representations of a word hyperbolic group into a semisimple Lie group with strong dynamical and geometric properties. Introduced by Labourie \cite{MR2221137} in the study of Hitchin representations, their theory was further developed by Guichard and Wienhard \cite{guichard2012anosov}, as well as Kapovich, Leeb, and Porti \cite{MR3569569}, \cite{MR3888689}, who provided a comprehensive framework and structural results. These representations admit well-behaved limit maps and satisfy stability properties, making them a natural higher-rank analogue of classical Teichmüller theory.
    
The dynamics of Anosov representations can be studied via the Gromov geodesic flow associated to a word hyperbolic group $\Gamma$. This flow space, which generalizes the unit tangent bundle of a negatively curved manifold, provides a natural setting to encode the asymptotic geometry of $\Gamma$. Via the equivariant limit maps of an Anosov representation, one can pull back the flow dynamics into associated bundles, allowing the definition of analytic invariants such as entropy, intersection, and the pressure metric.  In the following we define such representations in our setting. We follow the definition from \cite{ghosh2023margulis}.
    
        Let $\Gamma$ be a word hyperbolic group and $\partial_\infty \Gamma$ be its
boundary at infinity. The group $\Gamma$ acts naturally on
$\partial_\infty \Gamma$. Let $\partial_\infty \Gamma^{(2)}$ be the
subset of $\partial_\infty \Gamma \times \partial_\infty \Gamma$
obtained by removing the diagonal. 
Gromov~\cite{gromov1987hyperbolic} constructed a proper and cocompact action of $\Gamma$
on $\partial_\infty \Gamma^{(2)} \times \bR$, which extends the diagonal
action of $\Gamma$ on $\partial_\infty \Gamma^{(2)}$. 
The space $\partial_\infty \Gamma^{(2)} \times \bR$, equipped with this
action, will be denoted by $\widetilde{\U\Gamma}$, and we refer to the
action as the $\Gamma$–action on $\partial_\infty \Gamma^{(2)} \times
\bR$. 
The $\Gamma$–action commutes with the $\bR$–action
\[
(x,y,s) \longmapsto (x,y,s+t).
\]
The \emph{Gromov flow space} of $\Gamma$, denoted by $\U\Gamma$, is the
quotient space $\widetilde{\U\Gamma} / \Gamma$. 
We denote the flow on $\U\Gamma$ by $\varphi = (\varphi_t : \U\Gamma \to \U\Gamma)$.

Let $\sG^\tL$ be a semisimple Lie group and $(\sP_+,\sP_-)$ a pair of opposite
parabolic subgroups of $\sG^\tL$. Then the quotients $\sG^\tL/\sP_\pm$ are compact
homogeneous spaces. We consider the homogeneous space $\sG^\tL/\sL$, where
$\sL := \sP_+ \cap \sP_-$. If $\sG^\tL$ acts diagonally on $\sG^\tL/\sP_+ \times
\sG^\tL/\sP_-$, then $\sL$ is the stabilizer of $(e\sP_+, e\sP_-)$, and we identify
$\sG^\tL/\sL$ with an open orbit in $\sG^\tL/\sP_+ \times \sG^\tL/\sP_-$.

\begin{dfn}
    An injective homomorphism $\rho: \Gamma \to \sG^\tL $ is said to be  $(\sP^+,\sP^-)$-Anosov if and only if the following conditions hold: 

    $(1)$ There exists continuous, injective, $\rho(\Gamma)$-equivariant Anosov maps 
    $$\xi^\pm: \partial_{\infty}\Gamma \to \sG^\tL/\sP^\pm$$ such that $\xi(p):=(\xi^+(p_+), \xi^-(p_-)) \in \sG^\tL/\sL$ for any $p:= (p_+, p_-, t) \in \U\Gamma$.

    $(2)$ There exist positive constants $c, C$ and for $p \in \widetilde{\U\Gamma}$ a collection of Euclidean norm $\lVert \cdot\lVert$ on $T_{\xi(p)}\sG/\sL$ such that for all $\gamma \in \Gamma$ and $v \in T_{\xi(p)}\sG^\tL/\sL$ we have $\lVert \rho(\gamma)v\rVert_{\gamma p} = \lVert v \rVert_{\gamma}$ and for all  $v^+ \in T_{\xi^+(p_+)}\sG^\tL/\sP^+$ (resp. $v^- \in T_{\xi^-(p_-)}\sG^\tL/\sP^-$) and $t\geq 0$: $$\lVert v^+\rVert_{\phi_{t} p} \leq C e^{-ct} \lVert v^+\rVert_p \ \ (\text{resp.} \lVert v^-\rVert_{\phi_{-t} p} \leq C e^{-ct} \lVert v^+\rVert_p).$$
\end{dfn}

\begin{dfn}
We denote the space of all representations $\rho: \Gamma \to \sG^\tL$ by $\Hom(\Gamma, \sG^\tL)$, the space all $(\sP^+, \sP^-)$-Anosov representations by $\Hom(\Gamma, \sG^\tL, \sP^{\pm})$, and the space of space of Zariski dense Anosov representations by $\Hom^\Z(\Gamma, \sG^\tL, \sP^\pm)$.
\end{dfn}

\begin{thm}[{\cite{guichard2012anosov}}]
    The set $\Hom(\Gamma, \sG^\tL, \sP^\pm)$ is open in  $\Hom(\Gamma, \sG^\tL)$. 
\end{thm}

We space of all $\upsilon$-H\"older continuous maps from $\partial_{\infty}\Gamma$ to $\sG^\tL/\sP^\pm$ by $\mathcal{C}^{\upsilon}(\partial_{\infty}\Gamma, \sG^\tL/\sP^{\pm})$.

\begin{thm}[{\cite[Theorem~6.1]{bridgeman2015pressure}}]
    Let $\rho_0 \in \Hom(\Gamma, \sG^\tL, \sP^\pm)$. Then there exist a open neighbourhood of $U_0 \subset \Hom(\Gamma, \sG^\tL, \sP^\pm)$  of $\rho_0$ such that, for all $\rho \in U_0$, the associated  Anosov maps $\xi_{\rho}^\pm \in \mathcal{C}^{\upsilon}(\partial_{\infty}\Gamma, \sG^\tL/\sP^\pm)$ for some $\alpha >0$. Moreover, the map 
    \begin{align*}
        U_0 &\to \mathcal{C}^{\upsilon}(\partial_{\infty}\Gamma, \sG^\tL/\sP^\pm)\\
        \rho &\mapsto \xi^{\pm}_{\rho} 
    \end{align*}
    is real-analytic.
\end{thm}

\begin{thm}[{\cite{bridgeman2015pressure}}] \label{Th:MetricAnosov}
    Let $\rho:\Gamma \to \sG^\tL$ be  an Anosov representation, then the associated flow $\phi: \U\Gamma \to \U\Gamma$ is a topologically transitive metric Anosov flow.
\end{thm}

\subsection{Margulis spacetimes}

\begin{lemma}\label{Rmk: Conjugacy invar Margulis}
    For an infinite order element $\gamma$, and $g \in \sG$, we have $\alpha_{g\rho g^{-1}}(\gamma) = \alpha_{\rho}(\gamma)$. 

    \end{lemma}
    \begin{proof}
        For $g = (h,v)$, and $\rho = (\tL_{\rho}, u_{\rho})$, we have 
    \[
    (g\rho g^{-1})(\gamma) = g\rho(\gamma) g^{-1} = (h\tL_{\rho}(\gamma)h^{-1}, hu_{\rho}(\gamma) +v -h\tL_{\rho}(\gamma)h^{-1}v).
    \]
    Let $\widetilde{ \rho}: = g\rho g^{-1}$. Now, the limit maps for $\tL_{\widetilde{\rho}}$ are given by $\xi^\pm_{\tL_{\widetilde{\rho}}} = h \cdot\xi^\pm_{\tL_{\rho}}$. Moreover, we have $\xi_{\tL_{\widetilde{\rho}}}(\gamma^+, \gamma^-) = h \cdot \xi_{\tL_{\rho}}(\gamma^+, \gamma^-)$. Since the neutral map $\nu$ is $\sG$-equivariant, 
    \[
    \nu_{\widetilde{\rho}}(\gamma^+, \gamma^-) = \nu (h \cdot\xi_{\tL_{\rho}}(\gamma^+, \gamma^-)) =h \cdot \nu(\xi_{\tL_{\rho}}(\gamma^+, \gamma^-) = h \cdot \nu_{\rho}(\gamma^+, \gamma^-).
    \]
    Thus, we have 
    \begin{align*}
    \alpha_{\widetilde{\rho}} (\gamma) = \langle u_{\widetilde{\rho}}(\gamma)\mid \nu_{\widetilde{\rho}}(\gamma^+, \gamma^-)\rangle &= \langle hu_{\rho}(\gamma) +v -h\tL_{\rho}(\gamma)h^{-1}v \mid h \cdot \nu_{\rho}(\gamma^+, \gamma^-)\rangle\\
    &= \langle u_{\rho}(\gamma) +h^{-1}v -\tL_{\rho}(\gamma)h^{-1}v \mid  \nu_{\rho}(\gamma^+, \gamma^-)\rangle\\
    &= \langle u_{\rho}(\gamma)\mid \nu_{\rho}(\gamma^+, \gamma^-)\rangle\\
    &= \alpha_{\rho}(\gamma).
    \end{align*}
This establishes the result.    
\end{proof}


We now give definition of diffused version of Margulis invariant.

\begin{dfn}[ {\cite[Definition~1.15]{ghosh2023margulis}}]\label{Def:LM}
    Suppose    $\rho: \Gamma \to \sG$ be an injective homomorphism such that $\rho \in \Hom^\Z(\Gamma, \sG^\tL, \sP^\pm)$. Then the \emph{Labourie-Margulis invariant} of $\rho$ is a Livsic cohomology class $[f_{\rho}]$ of Holder continuous function $f_{\rho}$ such that $$\int f_{\rho} d\mu_{\gamma} = \frac{\alpha_{\rho}(\gamma)}{\ell_{\gamma}},$$
    where $\mu_{\gamma}$ is a flow invariant probability measure supported on the periodic orbit of $\U\Gamma$ corresponding to $\gamma$ and $\ell_{\gamma}$ is the period of this orbit.
\end{dfn}

 \begin{thm}[\cite{ghosh2023affine}]\label{Opposite Sign Lemma}
    Suppose $\tL_{\rho} \in \Hom^\Z(\Gamma, \sG^\tL, \sP^{\pm})$. Then $\rho$ is Margulis spacetime  if and only if Labourie-Margulis invariant is either positive or negative.
\end{thm}

\section{Thurston's metric on Margulis spacetimes}\label{Section 3}

\subsection{Thurston's asymmetric metric on Margulis spacetimes}

In this section, we will construct Thurston's asymmetric metric on Margulis spacetimes as constructed in the works of  Abels–Margulis–Soifer \cite{abels2002zariski}, Goldman-Labourie-Margulis \cite{goldman2009proper}, Ghosh-Trieb \cite{ghosh2023margulis}.

Consider the map introduced in the Section~\ref{Def: Pull back},
\[
\tR: \M(\Gamma, \sG) \to \bP\HR(\phi).
\]
 In this subsection we pullback the asymmetric distance of $\bP\HR(\phi)$ (See Section~3 of  Carvajales-Dai-Pozzetti-Wienhard \cite{carvajales2024thurston} for details) to the moduli space $\M(\Gamma, \sG)$. 

\begin{dfn}\label{Def:Asym}
    Define  $d_\Th: \M(\Gamma, \sG) \times \M(\Gamma, \sG) \rightarrow \bR \cup \{\infty\}$ by  $$d_{\Th}(\rho, \varrho) := \log\left(\sup_{[\gamma] \in [\Gamma_\tH]} \frac{h_{\rho}}{h_{\varrho}} \frac{\alpha_{\rho}(\gamma)}{\alpha_{\varrho}(\gamma)}\right).$$  
Here $\Gamma_{\tH} \subset \Gamma$ is the  subset  of infinite-order elements. For $\gamma \in \Gamma$, let $[\gamma]$ denotes its conjugacy class, and write $[\Gamma_{\tH}]$ for the set of conjugacy classes of infinite-order elements in $\Gamma$. 
\end{dfn}

\begin{rem}\label{Rem: Positive}
    The definition of $d_\Th$ is well defined since for $\rho, \varrho \in \Hom_\M(\Gamma, \sG^0 \ltimes \bR^{2n+1})$ the Margulis invariants $\alpha_{\rho} (\gamma)$ and $\alpha_{\varrho}(\gamma)$ are positive (See Ghosh~\cite{ghosh2018pressure}), which ensures that the logarithm is well defined. 
\end{rem}

\begin{lemma}\label{Lemma:Th 1}
The map $\tR$ preserves the Thurston's asymmetric metric, i.e., 
    \[
    d_\Th (\rho, \varrho) = d_\Th (\tR(\rho), \tR(\varrho))
    \]
\end{lemma}

\begin{proof}
Let $\gamma \in O$, and denote by $p_{\phi^{h_{\rho}f_{\rho}}}(\gamma)$ the period of the reparametrized flow $\phi^{h_{\rho}f_{\rho}}$. Then
    \begin{align*}
    d_\Th(\tR(\rho), \tR(\varrho)) &= \log \left(\sup_{[\gamma] \in [\Gamma_\tH]}\frac{h_\Top(\phi^{h_{\rho}f_{\rho}})}{h_\Top(\phi^{h_{\varrho}f_{\varrho}})}\frac{p_{\phi^{h_{\rho}f_{\rho}}} (\gamma)}{p_{\phi^{h_{\varrho}f_{\varrho}}} (\gamma)}\right)\\
    &=\log \left(\sup_{[\gamma] \in [\Gamma_\tH]}\frac{h_\Top(\phi^{h_{\rho}f_{\rho}})}{h_\Top(\phi^{h_{\varrho}f_{\varrho}})} \frac{\int_0^{\ell(\gamma)} h_{\rho}f_{\rho} (\phi_t(x)) dt}{\int_0^{\ell(\gamma)} h_{\varrho}f_{\varrho} (\phi_t(x)) dt}\right)\\
    &=\log \left(\sup_{[\gamma] \in [\Gamma_\tH]}\frac{h_\Top(\phi^{h_{\rho}f_{\rho}})}{h_\Top(\phi^{h_{\varrho}f_{\varrho}})} \frac{h_{\rho}\int_0^{\ell(\gamma)} f_{\rho} (\phi_t(x)) dt}{h_{\varrho}\int_0^{\ell(\gamma)} f_{\varrho} (\phi_t(x)) dt}\right)
    \end{align*}
    Finally, since there is a one-to-one correspondence between periodic orbits $O$ of $\phi$ and conjugacy classes of infinite-order elements, the  proof is complete.
\end{proof}

\begin{lemma}
    Let $\rho, \varrho \in \M(\Gamma, \sG)$, then $$\tR(\rho) = \tR(\varrho)  \Leftrightarrow h_{\rho} \tM_{\rho}= h_{\varrho} \tM_{\varrho}. $$
\end{lemma}
\begin{proof}
Let $\gamma \in O$, and denote by $p_{\phi^{h_{\rho}f_{\rho}}}(\gamma)$ the period of the reparametrized flow $\phi^{h_{\rho}f_{\rho}}$. Then by Lemma~\ref{Lem:Ther-Equ}, we have
\begin{align*}
\tR(\rho) = \tR(\varrho) &\Leftrightarrow h_\Top(\phi^{h_{\rho}f_{\rho}}) p_{\phi^{h_{\rho}f_{\rho}}}(\gamma)= h_\Top(\phi^{h_{\varrho}f_{\varrho}}) p_{\phi^{h_{\varrho}f_{\varrho}}}(\gamma)\\
& \Leftrightarrow  h_{\rho} \alpha_{\rho}(\gamma) = h_{\varrho}\alpha_{\varrho}(\gamma).
\end{align*}
Hence, the proof is complete.
\end{proof}

\begin{cor}\label{Thm:Thasm}
    The function $\sd_{\Th}(\cdot, \cdot)$ is real-valued, non-negative, and satisfies the triangle inequality. Furthermore, 
    \[d_{\mathrm{Th}}(\rho, \varrho) = 0 \Leftrightarrow h_{\rho}\alpha_{\rho} = h_{\varrho}\alpha_{\varrho}.\]
\end{cor}

\begin{proof}
This is an immediate consequence of Lemma~\ref{Lemma:Th 1}, the preceding lemma, and Theorem~\ref{Th:Thasm-Therm}.
\end{proof}

\subsection{Renormalized marked Margulis invariant spectrum rigidity}

\begin{thm}\label{Th: fixed }
    Let $\rho, \varrho: \Gamma \to \sG$ be two Margulis spacetimes such that the linear parts are fixed. Then if the equality $h_{\rho}\alpha_{\rho} = h_{\varrho}\alpha_{\varrho}$ holds, there exists an isomorphism $\sigma: \sG \to \sG$ such that $\sigma \circ \rho = \varrho$.
\end{thm}

\begin{proof}
    Let $\rho = (\tL_{\rho},u_{\rho})$ and $\varrho = (\tL_{\varrho}, u_{\varrho})$ are two Margulis spacetimes such that $\tL_{\rho} = \tL_{\varrho} = \tL$. Let $\widetilde{\rho} = (\tL_{\widetilde{\rho}}, u_{\widetilde{\rho}})$ and $\widetilde{\varrho}=(\tL_{\widetilde{\varrho}}, u_{\widetilde{\varrho}})$ be two Margulis spacetimes defined by 
    \[
    \widetilde{\rho} := (\tL, h_{\rho}u_{\rho}), \quad  \widetilde{\varrho} := (\tL, h_{\varrho} u_{\varrho}).
    \]
   Suppose $h_{\widetilde{\rho}}$ (resp. $h_{\widetilde{\varrho}}$) denotes the topological entropy of $\widetilde{\rho}$ (resp. $\widetilde{\varrho}$). Then by \cite[Lemma 3.2]{ghosh2023margulis} $h_{\widetilde{\rho}} = h_{\widetilde{\varrho}} = 1$. Therefore, $\widetilde{\rho}, \widetilde{\varrho} \in \M_1(\Gamma, \sG)$. Moreover, scaling property of Margulis invariants, we have 
    \[
    \alpha_{\widetilde{\rho}} (\gamma) = h_{\rho}\alpha_{\rho} (\gamma) = h_{\varrho}\alpha_{\varrho} (\gamma) =\alpha_{\widetilde{\varrho}} (\gamma), \quad \text{for all } \gamma \in \Gamma.
    \]
    In  $\M_1(\Gamma, \sG)$, we define a  path $\rho_t = (\tL_{\rho_t}, u_{\rho_t})$
    \[
    \rho_{t}: = (\tL, (1-t)u_{\widetilde{\rho}} + t u_{\widetilde{\varrho}}), \quad \text{for } t \in [0,1].
    \]
    Then $\rho_0 = \widetilde{\rho}$, and $\rho_1 = \widetilde{\varrho}$. We note that since the linear part of $\rho_t$ is fixed, then the unit $1$-eigenvector  of $\tL_{\rho_t}(\gamma)$ is same for all $t$. Then by the linearity of Margulis invariant, we have 
    \[
    \alpha_{\rho_t} (\gamma) = (1-t)\alpha_{\widetilde{\rho}} (\gamma) + t \alpha_{\widetilde{\varrho}}(\gamma).
    \]
    Now since $\alpha_{\widetilde{\rho}} (\gamma) =\alpha_{\widetilde{\varrho}}(\gamma)$, we have $\alpha_{\rho_t} (\gamma) = \alpha_{\rho} (\gamma)$.   Thus, we have 
    \[
    U(\gamma) := \frac{d}{dt}\bigg|_{t=0} \rho_t (\gamma) \widetilde{\rho}(\gamma)^{-1} =  (0, u_{\widetilde{\varrho}}(\gamma) - u_{\widetilde{\rho}}(\gamma)) .
    \]
    Note that $U(\gamma) \in Z^1_{Ad \circ \widetilde{\rho}} (\Gamma, \mathesstixfrak{g}^\tL \ltimes \bR^{2n+1})$. Indeed, by a straightforward calculation shows that 
     \[
     U(\gamma_1 \gamma_2) = \frac{d}{dt}\bigg|_{t=0} \rho_t (\gamma_1 \gamma_2) \widetilde{\rho}(\gamma_1 \gamma_2) = Ad (\widetilde{\rho}(\gamma_1) U(\gamma_2) + U(\gamma_1).
     \]
     Since $\frac{d}{dt}\big|_{t=0}\alpha_{\rho_t}(\gamma) = 0$, by \cite[Theorem 2.1]{ghosh2023margulis} $U(\gamma) \in B^1_{Ad \circ \widetilde{\rho}}(\Gamma, \mathesstixfrak{g}^\tL \times \bR^{2n+1})$. Thus, there exists $G \in \mathesstixfrak{g}^\tL \ltimes \bR^{2n+1}$ such that $U(\gamma) = G - Ad(\widetilde{\rho}(\gamma))G$. Let $G = (H,v) \in \fk g^\tL \ltimes \bR^{2n+1}$, then 
     \[
     Ad(\widetilde{\rho}(\gamma)) G = Ad(\tL(\gamma), u_{\widetilde{\rho}})(X,v) = (\tL(\gamma)X\tL(\gamma)^{-1}, \tL(\gamma) v- (\tL(\gamma)X\tL(\gamma)^{-1})u_{\widetilde{\rho}}(\gamma)).
     \]
     Therefore, 
     \[
     G-Ad(\widetilde{\rho}(\gamma))G = (X-\tL(\gamma)X\tL(\gamma)^{-1}, v- \tL(\gamma)v+(\tL(\gamma)X\tL(\gamma)^{-1})u_{\widetilde{\rho}(\gamma)}).
     \]
     Solving the equation $U(\gamma) = G-Ad(\widetilde{\rho}(\gamma))G$, yields $X= \tL(\gamma)X\tL(\gamma)^{-1}$. Since $\tL(\Gamma)$ is Zariski dense in $\sG^\tL$, and $\sG^\tL$ is simple, this forces $X=0$. Therefore, there exists $v \in \bR^{2n+1}$ such that 
    \[
    v - \rho(\gamma)v  = u_{\widetilde{\varrho}}(\gamma) - u_{\widetilde{\rho}}(\gamma). 
    \]
    Now if we take $g: = (I, v)$, then 
    \begin{align*}
        g\widetilde{\rho}(\gamma)g^{-1} &= (I,v)(\tL(\gamma), u_{\widetilde{\rho}}(\gamma))(I, -v)\\
        &= (\tL(\gamma),u_{\widetilde{\rho}}(\gamma) + v-\tL(\gamma)v)\\
        &= (\tL(\gamma), u_{\widetilde{\rho}}(\gamma) + u_{\widetilde{\varrho}}(\gamma) - u_{\widetilde{\rho}}(\gamma))\\
        & = (\tL(\gamma), u_{\widetilde{\varrho}}(\gamma)\\
        &= \widetilde{\varrho} (\gamma).
    \end{align*}
    Let $\phi_{\lambda} \in \mathrm{Aut}(\sG)$ defined by $\phi_{\lambda}(A,u) = (A, \lambda u)$. Then, we have 
    \[
    \phi_{h_{\rho}}(\rho) = \widetilde{\rho}, \quad \phi_{h_{\varrho}}(\rho) = \widetilde{\varrho}
    \]
    We define $\sigma: = \phi_{h_{\varrho}^{-1}} \circ Ad_g \circ \phi_{h_{\rho}} \in \mathrm{Aut}(\sG)$, and see that 
    \begin{align*}
        \sigma \circ \rho(\gamma) &= \phi_{h_{\varrho}^{-1}} (Ad_g \circ \phi_{h_{\rho}} (\rho(\gamma))\\
        &= \phi_{h_{\varrho}^{-1}} (g \phi_{h_{\rho}} (\rho(\gamma)g^{-1})\\
        &= \phi_{h_{\varrho}^{-1}} (g \widetilde{\rho}(\gamma)g^{-1})\\
        &= \phi_{h_{\varrho}^{-1}}(\widetilde{\varrho}(\gamma))\\
        &=\varrho(\gamma).
    \end{align*}
    Therefore, we get $\sigma \circ \rho = \varrho$.
\end{proof}

\subsection{Proof of Theorem~\ref{Th:Rigidity}} We need the following result from \cite{kim2005affine}
 and \cite{ghosh2020isospectrality} to prove the theorem in its general form.

\begin{thm} \label{Thm:Kim}   Let $\rho, \varrho: \Gamma \to \sG$ be two Margulis spacetimes. Then if the equality $\alpha_{\rho} = \alpha_{\varrho}$ holds, there exists an isomorphism $\sigma: \sG \to \sG$ such that $\sigma \circ \rho = \varrho$. 
\end{thm}

Now, suppose $\rho=(\tL_{\rho}, u_{\rho})$ and $\varrho = (\tL_{\varrho}, u_{\varrho})$ be two Margulis spacetimes, with corresponding  entropies  $h_{\rho}$ and $h_{\varrho}$. We also consider $\widetilde{\rho} = (\tL_{\widetilde{\rho}}, u_{\widetilde{\rho}})$ and $\widetilde{\varrho}=(\tL_{\widetilde{\varrho}}, u_{\widetilde{\varrho}})$ be two Margulis spacetimes whose  translation parts are given by $u_{\widetilde{\rho}} = h_{\rho}u_{\rho}$, and $u_{\widetilde{\varrho}} = h_{\varrho}u_{\varrho}$. Then, we have 
     \[
    \alpha_{\widetilde{\rho}} (\gamma) = h_{\rho}\alpha_{\rho} (\gamma) = h_{\varrho}\alpha_{\varrho} (\gamma) =\alpha_{\widetilde{\varrho}} (\gamma), \quad \text{for all } \gamma \in \Gamma.
    \]
    By Theorem~\ref{Thm:Kim}, there exist $g \in \sG$ such that 
    $g \widetilde{\rho}(\gamma)g^{-1} = \widetilde{\varrho}(\gamma)$ for all $\gamma\in \Gamma$.     Let $\phi_{\lambda} \in \mathrm{Aut}(\sG)$ defined by $\phi_{\lambda}(h,u) = (h, \lambda u)$. Then, we have  $\phi_{h_{\rho}}(\rho) = \widetilde{\rho}$, and   $\phi_{h_{\varrho}}(\rho) = \widetilde{\varrho}$. We define $\sigma: = \phi_{h_{\varrho}^{-1}} \circ Ad_g \circ \phi_{h_{\rho}} \in \mathrm{Aut}(\sG)$, and see that 
    \begin{align*}
        \sigma \circ \rho &= \phi_{h_{\varrho}^{-1}} (Ad_g \circ \phi_{h_{\rho}} (\rho))\\
        &= \phi_{h_{\varrho}^{-1}} (g \phi_{h_{\rho}} (\rho) g^{-1})\\
        &= \phi_{h_{\varrho}^{-1}} (g \widetilde{\rho}g^{-1})\\
        &= \phi_{h_{\varrho}^{-1}}(\widetilde{\varrho})\\
        &=\varrho.
    \end{align*}
    Therefore, we get $\sigma \circ \rho = \varrho$.
\qed

\section{Finsler norm on Margulis spacetimes}\label{Section 4}

In this section, we will use the the map $\tR$ (See Subsection~\ref{Def: Pull back}) to pull back the Finsler norm on $\bP\HR^v(\phi)$ to $\M(\Gamma, \sG)$.

A family of representation $\{\rho: \Gamma \rightarrow \sG\}$ is parametrized by a real analytic disk $D$ is \emph{real analytic} if for all $\gamma \in \Gamma$, the map $z \mapsto \rho_z (\gamma)$ is real analytic. We fix a real  analytic neighbourhood of $\rho \in  \M(\Gamma, \sG)$ and a real analytic family $\{\rho_z\}_{z \in D} \subset \M(\Gamma, \sG)$, parametrized by some real analytic disk $D$ around $0$ so that $\rho_0 =\rho$ and $\cup_{z \in D} \rho_z$ coincides with this neighbourhood. By abuse of notation we will sometimes identify the neighbourhood with $D$ itself.

\begin{dfn}\label{Def:Finsler norm}
    Given a tangent vector $v \in \rT_{\rho} \M(\Gamma, \sG)$ we set 
    \[
    \lVert v\rVert_\Th := \sup_{[\gamma] \in [\Gamma_\tH]} \frac{
       \dot{h}_{\rho}(v)\,\alpha_\rho(\gamma)
       + h_\rho\, \dot{\alpha}_{\rho}(\gamma))(v)
     }{h_\rho\,\alpha_\rho(\gamma)},
    \]
    where $\dot{h}_{\rho}$ (resp. $\dot{\alpha}_{\rho}(\gamma)$ is the  derivative of $\varrho \mapsto h_{\varrho}$ (resp. $\varrho \mapsto \alpha_{\varrho} (\gamma)$) at $\varrho$. In particular, if $\varrho \mapsto h_{\varrho}$ is constant, that is, if $\varrho \in \M_k(\Gamma, \sG)$, then one has
    \[
    \lVert v \rVert_\Th = \sup_{[\gamma] \in [\Gamma_\tH]} \frac{\dot{\alpha}_{\rho} (\gamma) (v)}{\alpha_{\rho}(\gamma)}.
    \]
\end{dfn}

\begin{rem}
    $(1)$ By~\cite[Section~3]{ghosh2023margulis}, entropy varies in an analytic way over $\M(\Gamma, \sG)$. In particular, the map $\rho \mapsto h_{\rho}$ is differentiable.
    
    $(2)$ By~\cite[Section~2]{ghosh2023margulis}, the Margulis invariant varies analytically over moduli spaces. In particuar, the map $\rho \mapsto \alpha_{\rho}(\gamma)$ is differentiable.
\end{rem}

The next proposition generalizes \cite[Proposition 4.2.1]{ghosh2018pressure} to the setting of higher-dimensional Margulis spacetimes. The proof follows the spirit of Ghosh’s original argument.

\begin{prop}\label{Prop:Main Finsler}
    Let $\{\rho_z\}_{z \in D} \subset \Hom_\M (\Gamma, \sG)$  be a real analytic family of Margulis spacetimes parametrized by an open disk $D$ around $0$. Then  after possibly restricting $D$ to a smaller disk around $0$, there exists  $\upsilon >0$, and a  real analytic family of positive H\"older continuous functions $\{\widetilde{g}_z: \U\Gamma \rightarrow \bR\}_{z \in D} \subset \mathcal{H}^{\upsilon}(\U\Gamma)$ such that for each $z \in D$, 
    \[
    \int_{0}^{\ell(\gamma)} \widetilde{g}_z(\phi_t(x)) dt = \alpha_{\rho_z}(\gamma).
    \]
    Moreover, the map $D \to \bP\HR(\phi)$, defined by $z \mapsto  [\phi^{\widetilde{g}_z}]$ is real analytic.
\end{prop}

We shall divide the proof of the Proposition~\ref{Prop:Main Finsler} into the following two lemmas.

\begin{lemma}\label{Lem:contracting_section}
Under the same hypothesis as in Proposition~\ref{Prop:Main Finsler}, after possibly shrinking $D$ to a smaller disk around $0$, there exist $t_0>0$, $\upsilon>0$, and an 
$\upsilon$–H\"{o}lder, transversely real analytic section 
\[
\sigma : D \times \U\Gamma \to E
\]
of the affine bundle 
\[
\pi : E=\bigl(D\times\widetilde{\U\Gamma}\times\mathbb{R}^{2n+1}\bigr)/\Gamma 
     \longrightarrow D\times \U\Gamma
\]
such that:
\begin{enumerate}
    \item $\Psi_{t_0}(\sigma(z,p))=\sigma(z,\psi_{t_0}(p))$,
    \item $\sigma(0,p)=\sigma_0(p)$, where $\sigma_0$ is a neutralized section for $\rho_0$,
    \item $\Psi_{t_0}$ is uniformly contracting along the fibers of $E$ restricted to $\sigma$.
\end{enumerate}
\end{lemma}

\begin{proof}
       We define the trivial affine bundle over $D \times \widetilde{\U\Gamma}$ by  
\[
\widetilde{E}= D \times \widetilde{\mathsf{U}\Gamma} \times \mathbb{R}^{2n+1}.
\]
The group $\Gamma$ acts on $\widetilde{E}$ via  
\[
\gamma \cdot (u, \tilde{p}, v)= (u, \gamma \cdot \tilde{p}, \rho_z(\gamma)v),
\]
where $\gamma \cdot \tilde{p}$ denotes the diagonal action on $\partial_\infty \Gamma^{(2)} \times \mathbb{R}$.  
The quotient bundle is
\[
\pi : E =  \widetilde{E}/\Gamma \longrightarrow D \times \mathsf{U}\Gamma .
\]
The bundle $E$ is Lipschitz and transversely real analytic (see Definition~6.4 of \cite{bridgeman2015pressure}).  

Fix $t_0>0$. Define the flow
\[
\Psi_{t_0}(z,\widetilde{p},v) = (z,\widetilde{\psi}_{t_0}(\widetilde{p}),v).
\]
Since $\Psi_{t_0}$ is $\Gamma$–equivariant, it descends to a map, still denoted $\Psi_{t_0}$,
\[
\Psi_{t_0}\bigl([z,\widetilde{p},v]\bigr)= [z,\widetilde{\psi}_{t_0}(\widetilde{p}),v],
\]
where $[\cdot]$ denotes the $\Gamma$–orbit.  
The Gromov flow is $\widetilde{\psi}_{t_0}(x,y,t)=(x,y,t+t_0)$ on $\widetilde{\U\Gamma}$.

Let $\pi_{\Gamma} : \widetilde{\U\Gamma} \to \U\Gamma$ be the projection $\pi_{\Gamma}(\widetilde{p})=\Gamma\cdot \widetilde{p}$.  
Since $\widetilde{\psi}_{t_0}$ commutes with the $\Gamma$–action,
\[
\widetilde{\psi}_{t_0}(\gamma \cdot \widetilde{p}) = \gamma \cdot \widetilde{\psi}_{t_0}(\widetilde{p}),
\]
the flow descends to a flow $\psi_{t_0}$ on $\U\Gamma$ satisfying  
\[
\pi_{\Gamma}\circ \widetilde{\psi}_{t_0} = \psi_{t_0}\circ \pi_{\Gamma}.
\]

For $[z,\widetilde{p},v]\in E$, the base point is  
\[
\pi([z,\widetilde{p},v]) = (z,\pi_{\Gamma}(\widetilde{p}))=(z,p),
\]
and
\[
\pi\bigl(\Psi_{t_0}([z,\widetilde{p},v])\bigr)
   = (z,\psi_{t_0}(p)).
\]
Thus $\Psi_{t_0}$ sends the fiber over $(z,p)$ to the fiber over $(z,\psi_{t_0}(p))$, and it covers $\mathrm{id}\times \psi_{t_0}$. Hence $\Psi_{t_0}$ is a bundle automorphism of $E$.

By \cite[Proposition~5.2]{ghosh2023affine}, for $\rho_0$ there exists a neutralized section
\[
f_0:\U\Gamma\to \mathrm{R}_{\rho_0}.
\]
Choose a lift $\widetilde{f}_0$ and define  
\[
\widetilde{\sigma}_0(\gamma\cdot\widetilde{p})=\rho_0(\gamma)\widetilde{\sigma}_0(\widetilde{p}).
\]
Then a section $\sigma_0 : \U\Gamma\to E$ is given by  
\[
\sigma_0(p) = [0,\widetilde{p},\widetilde{\sigma}_0(\widetilde{p})], \qquad 
\widetilde{p}\in \pi_{\Gamma}^{-1}(p).
\]

We check that $\Psi_{t_0}$ fixes $\sigma_0$. Indeed,
\[
\Psi_{t_0}(\sigma_0(p))
   = [0,\widetilde{\psi}_{t_0}(\widetilde{p}),\widetilde{\sigma}_0(\widetilde{p})],
\]
while
\[
\sigma_0(\psi_{t_0}(p))
   = [0,\widetilde{\psi}_{t_0}(\widetilde{p}),
        \widetilde{\sigma}_0(\widetilde{\psi}_{t_0}(\widetilde{p}))].
\]
Since $f_0$ is parallel along flow lines,
\[
\widetilde{\sigma}_0(\widetilde{\psi}_{t_0}(\widetilde{p}))
   = \widetilde{\sigma}_0(\widetilde{p}),
\]
and hence
\[
\Psi_{t_0}(\sigma_0(p))=\sigma_0(\psi_{t_0}(p)).
\]

By \cite[Definition~4.4]{ghosh2023affine}, there exist $C,c>0$ such that the fiberwise derivative
\[
T_{\sigma_0(p)}\pi^{-1}(p)\to T_{\sigma_0(\psi_{t_0}(p))}\pi^{-1}(\psi_{t_0}(p))
\]
satisfies
\[
\left\lVert D^f\Psi_{t_0}\Big|_{T_{\sigma_0(p)}\pi^{-1}(p)} \right\rVert
   \le Ce^{-ct_0} < 1 
   \qquad\text{(for $t_0$ large)}.
\]
Thus $\Psi_{t_0}$ contracts along $\sigma_0$.

Hence, by \cite[Theorem~6.5]{bridgeman2015pressure}, there exist an open disk $D_0\subset D$ about $0$, a constant $\upsilon>0$, and an $\upsilon$–Hölder transversely real analytic section  
\[
\sigma : D_0 \times \U\Gamma \to E
\]
such that  
\begin{enumerate}
\item $\Psi_{t_0}(\sigma(z,p)) = \sigma(z,\psi_{t_0}(p))$,
\item $\sigma|_{\{0\}\times \U\Gamma} = \sigma_0$,
\item $\Psi_{t_0}$ contracts along $\sigma$.
\end{enumerate}
This completes the proof of the lemma.
\end{proof}

\begin{lemma}\label{Lem:Holder_function}
Let $\sigma$ be the section obtained in Lemma~\ref{Lem:contracting_section}.
Then for each $z\in D$ there exists a positive $\upsilon$–Hölder function 
$\widetilde{g}_z:\U\Gamma\to\mathbb R$,
real analytically depending on $z$, such that for every infinite order
$\gamma\in\Gamma$,
\[
\int_0^{\ell(\gamma)} \widetilde{g}_z(\phi_t(x))\,dt
   = \alpha_{\rho_z}(\gamma).
\]
Moreover, the map $D \to \bP\HR(\phi)$, defined by $z\mapsto [\phi^{\widetilde{g}_z}]$ is real analytic.
\end{lemma}

\begin{proof}

Lift $\sigma$ to a $\Gamma$–equivariant section  
\[
\widetilde{\sigma}:D_0\times \widetilde{\U\Gamma}\to \mathbb{R}^{2n+1},
\qquad
\widetilde{\sigma}(z,\gamma\cdot \widetilde{p}) 
   = \rho_z(\gamma)\widetilde{\sigma}(z,\widetilde{p}).
\]
For each $z\in D_0$, set $\sigma_z(\widetilde{p})=\widetilde{\sigma}(z,\widetilde{p})$.

For $z\in D_0$ and $\widetilde{p}=(x,y,t)\in\widetilde{\U\Gamma}$, define $k_s(z,\widetilde{p})$ by
\[
\widetilde{\sigma}_z(\widetilde{\psi}_s(\widetilde{p}))
  = \widetilde{\sigma}_z(\widetilde{p})
    + k_s(z,\widetilde{p})\cdot 
      \nu(\widetilde{\sigma}_{\rL_{\rho_z}}(x,y)),
\]
where $\widetilde{\sigma}_{\rL_{\rho_z}}$ is the lift of the Anosov section for $\rL_{\rho_z}$.  
The cocycle property holds:
\[
k_{t+t'}(z,\widetilde{p})
   = k_t(z,\widetilde{\psi}_{t'}(\widetilde{p}))
     + k_{t'}(z,\widetilde{p}).
\]
The map $(z,\widetilde{p})\mapsto k_t(z,\widetilde{p})$ is $\upsilon$–Hölder in $\widetilde{p}$ and real analytic in $z$.

Fix $r>0$. Define
\[
\mathfrak{K}_t(z,\widetilde{p})
   = \log\!\left(
      \frac{\displaystyle\int_0^r \exp(k_{s+t}(z,\widetilde{p}))\,ds}
           {\displaystyle\int_0^r \exp(k_s(z,\widetilde{p}))\,ds}
     \right).
\]
Using the cocycle identity,  
\[
\mathfrak{K}_{t+t'}
   = \mathfrak{K}_t(\cdot,\widetilde{\psi}_{t'}(\cdot))
       + \mathfrak{K}_{t'}(\cdot).
\]
Define
\[
\mathfrak{f}_z(\widetilde{p})
  = \left.\frac{\partial}{\partial t}\right|_{t=0}
       \mathfrak{K}_t(z,\widetilde{p}).
\]

To compute explicitly $\fk f_z(\widetilde{p})$, we let $A_t = \displaystyle\int_0^r \exp(k_{s+t}(z,\widetilde{p}))\,ds$, and $A_0 = \displaystyle\int_0^r \exp(k_{s}(z,\widetilde{p}))\,ds$. Then, 
\[
\fk K(z,\widetilde{p}) = \log A_t - \log A_0.
\]
Differentiating with respect to $t$,  we get 
\[
\frac{\partial}{\partial t} \fk K_t(z, \widetilde{p})= \frac{A_t'}{A_t},
\]
where $A'_t= \frac{d}{dt}A_t$. More explicitly, 
\[
A_t' = \frac{d}{dt} \displaystyle \int_0^r \exp(k_{s+t}(z, \widetilde{p}) =\displaystyle \int_0^r \exp(k_{s+t}(z, \widetilde{p})) \cdot \frac{\partial}{\partial t} k_{s+t}(z, \widetilde{p}))\, ds, 
\]
and 
\[
A_0' = \displaystyle \int_0^r \exp(k_{s}(z, \widetilde{p})) \cdot \frac{\partial}{\partial t} k_{s}(z, \widetilde{p}))\, ds = \exp(k_s(z, \widetilde{p})) - \exp(k_0(z, \widetilde{p})).
\]
Since $k_0(z, \widetilde{p}) =0$, and $\exp(k_0(z,\widetilde{p}))= 1$, we have 
$$A_0'= \exp(k_r(z, \widetilde{p}))-1.$$
Therefore, 
\[
\fk f_z(\widetilde{p}) = \frac{A_0'}{A_0} =   \frac{\exp(k_r(z,\widetilde{p})) - 1}
      {\displaystyle\int_0^r \exp(k_s(z,\widetilde{p}))\,ds}.
\]

Note that $\mathfrak{f}_z$ is $\upsilon$–H\"{o}lder in $\widetilde{p}$ and real analytic in $z$. 

Let $\widetilde{p}_0 = (\gamma^-, \gamma^+,0)$. Then the element $\gamma \in \Gamma$ acts by $\gamma \cdot \widetilde{p}_0 = (\gamma^-, \gamma^+, t_{\gamma})$, where $t_{\gamma}$ is the period of $\gamma$. Since $\widetilde{\sigma}_u$ is $\Gamma$-equivariant, then $\widetilde{\sigma}_z(\gamma\cdot \widetilde{p}_0) = \rho_z(\gamma) \widetilde{\sigma}_z(\widetilde{p}_0)$. On the other hand, by the definition of $k_t$, we have 
\[
\widetilde{\sigma}_z(\gamma \cdot \widetilde{p}_0) = \widetilde{\sigma}_z(\widetilde{\psi}_{t_{\gamma}}(\widetilde{p}_0)) = \widetilde{\sigma}_z(\widetilde{p}_0) + k_{t_{\gamma}}(z, \widetilde{p}_0) \cdot\nu(\widetilde{\sigma}_{\tL_{\rho}}(\gamma^-, \gamma^+)).
\]
Also, we have $\rho_z(\gamma) \widetilde{\sigma}_z(\widetilde{p}_0) = \widetilde{\sigma}_z(\widetilde{p}_0) + \alpha_{\rho_z}(\gamma) \cdot \nu (\widetilde{\sigma}_{\tL_{\rho}}(\gamma^-, \gamma^+))$. Therefore, by comparing we get
$$k_{t_{\gamma}}(z, \widetilde{p}_0) = \alpha_{\rho_z}(\gamma).$$
Also, by the cocycle property for $t=t_{\gamma}$, we get 
\[
k_{s+t_{\gamma}}(z, \widetilde{p}_0) = k_s(z, \widetilde{p}_0) + \alpha_{\rho_z}(\gamma).
\]
Thus, $\fk K_t$ satisfies the following cocycle property:
\[
\fk K_{t+s} (z, \widetilde{p}) = \fk K_t (z, \widetilde{\psi}_t(\widetilde{p})) + \fk K_s(z, \widetilde{p}). 
\]
Differentiating with respect to $t$ at $t=0$, we get 
\[
\frac{\partial}{\partial t}\Bigg|_{t=0} \fk K_{t+s}(z, \widetilde{p}) = \frac{\partial}{\partial t}\Bigg|_{t=0} \fk K_{t}(z, \widetilde{\psi}_s(\widetilde{p}))= \fk f_z(\widetilde{\psi}_s(\widetilde{p})).
\]
Therefore, 
\[
\displaystyle \int_0^{t_{\gamma}} \fk f_z(\widetilde{\psi}_s(\widetilde{p}_0))\, ds = \fk K_{t_{\gamma}}(z, \widetilde{p}_0) - \fk K_0(z, \widetilde{p}_0) = \fk K_{t_{\gamma}}(z, \widetilde{p}_0) -1.
\]

For $\gamma\in\Gamma$ of infinite order with period $t_{\gamma}$,  
\[
\int_0^{t_{\gamma}}
   \mathfrak{f}_z(\gamma^{-},\gamma^{+},s)\,ds
   = {\mathfrak{K}}_{t_{\gamma}}(z,(\gamma^{-},\gamma^{+},0))
   = \alpha_{\rho_z}(\gamma).
\]
Since $k_{t+t_{\gamma}} = k_s + \alpha_{\rho_z}(\gamma)$,
\[
\displaystyle \int_0^r \exp(k_{s+t_{\gamma}})\, ds =\exp(\alpha_{\rho_z}(\gamma)) \int_0^r \exp(k_s)\, ds.
\]
Therefore, 
\[
\fk K_{t_{\gamma}}(z, \widetilde{p}_0) = \log(\exp(\alpha_{\rho_z}(\gamma))) = \alpha_{\rho_z}(\gamma). 
\]
Since $\widetilde{\psi}_s(\widetilde{p}_0) = (\gamma^-, \gamma^+, s)$, we have
\[
\displaystyle \int_0^{t_{\gamma}} \fk f_z(\gamma^-, \gamma^+, s)\, ds = \alpha_{\rho_z}(\gamma).
\]
Hence, by the Liv\v{s}ic theorem~\ref{thm:Livsic}, $\fk f_z$ is Liv\v{s}ic cohomologous to $\fk f_{\rho_z}$. Now using Lemmas A.1 and A.2 of \cite{goldman2012geodesics}, there exists an open disk $D_1 \subset D_0$, a positive real number $T$ such that for all $z \in D_1$
\[
\widetilde{g}_z(x,y,t_0) := \frac{1}{T} \displaystyle \int_0^T \fk f_z(x,y, t_0+s)\, ds >0.   
\]
This completes our proof.
\end{proof}

Let $\{\tilde{g}_z\}$ be a real analytic family  as in Proposition~\ref{Prop:Main Finsler}, then by \cite[Proposition~3.12]{bridgeman2015pressure}, the map  $z \mapsto h_{\phi^{\tilde{g}_z}}$ is real analytic. Therefore, $z \mapsto h_{\rho_z}$ is also real analytic.

    We observe that in Definition~\ref{Def:Finsler norm} it is sufficient to consider only  periodic orbits while taking the supremum, that is, 
     \[
    \lVert v\rVert_\Th = \sup_{\gamma \in O} \frac{
       \dot{h}_{\rho}(v)\,\alpha_\rho(\gamma)
       + h_\rho\, \dot{\alpha}_{\rho}(\gamma))(v)
     }{h_\rho\,\alpha_\rho(\gamma)}.
    \]
    Indeed, since  the periodic orbits  of the flow  are in one-to-one correspondence with conjugacy classes of infinite order elements of $\Gamma$.

\subsection{Proof of Theorem~\ref{Lemma:Tgnt Vec}}
    We begin by noting that, the derivative of $h_s$ (resp. $\alpha_s(\gamma)$) at $s = 0$ is 
    $\dot{h}_{\rho}(v) = \frac{d}{ds}\big|_{s=0} h_s$ (resp.$\dot{\alpha}_{\rho} (\gamma)(v) = \frac{d}{ds}\big|_{s=0} \alpha_{s}(\gamma)$). Then by Proposition~\ref{Prop:Main Finsler} we have
    \[\lVert v\rVert_\Th = \sup_{a \in O} \frac{d}{ds}\bigg|_{s=0} \frac{h_s}{h_{\rho}} \frac{\alpha_s(\gamma)}{\alpha_{\rho} (\gamma)} =  \sup_{a \in O} \frac{d}{ds}\bigg|_{s=0} \frac{h_s}{h_0} \frac{\int_{\U\Gamma} \tilde{g}_s d \mu_{\gamma}}{\int_{\U\Gamma} \tilde{g}_0 \sd \mu_{\gamma}}.
    \]
    Hence, 
    \[
    \lVert v\rVert_\Th = \sup_{a \in O} \frac{d}{ds}\bigg|_{s=0} \frac{\int_{\U\Gamma}g_s d\mu_{\gamma}}{\int_{\U\Gamma} g_0d\mu_{\gamma}}=  \sup_{a \in O} \frac{\int_{\U\Gamma} \dot{g}_0d \mu_{\gamma}}{\int_{\U\Gamma} g_0 d\mu_{\gamma}}.
    \]
Since periodic orbits are dense, the supremum over periodic orbits $O$ equals the supremum over all invariant probability measures. Combining this  with  Theorem~\ref{Thm:Fins1}, we obtain our desired conclusion.
\qed

\begin{cor}\label{Coro:derivative of thurston metric}
    Let $\{\rho_z\}_{z \in D} \subset \Hom_\M(\Gamma, \sG)$ be a real analytic family of Margulis spacetimes parameterizing a neighborhood of $\rho = \rho_0$. Fix an analytic path $z: (-1,1) \rightarrow D$ with $z(0) = 0$, and let $\rho_s = \rho_{z(s)}$ and $v = \frac{d}{d s}\big|_{s=0} \rho_s\rho^{-1}$ be the tangent vector at $\rho$. Then $s \mapsto d_\Th(\rho, \rho_s)$ is differentiable at $s=0$ and 
    \[\lVert v\rVert_{\Th} = \frac{d}{d s}\bigg|_{s=0} d_{\Th} (\rho, \rho_s).\]
\end{cor}

\begin{proof}
Observe that, if we set  $h_s := h_{\rho_z}$ and $g_s := h_s \tilde{g}_{z(s)}$, then $[\phi^{g_s}] = \left[\phi^{\tilde{g}_z}\right]$. It then follows from Lemma~\ref{Lemma:Th 1} and Proposition~\ref{Prop:Main Finsler}  that $d_\Th (\rho, \rho_s) = d_\Th \left(\left[\phi^{\tilde{g}_0}\right], \left[\phi^{\tilde{g}_z}\right]\right)$. Therefore, by  Theorem~\ref{Thm:Fins1} and Lemma~\ref{Lemma:Tgnt Vec}, the map $s \mapsto d_\Th(\rho, \rho_s)$ is differentiable at $s=0$ and 
\[
\frac{d}{d s}\bigg|_{s=0} d_{\Th} (\rho, \rho_s) = \lVert v\rVert_\Th.  
\]
This completes the proof. 
\end{proof}

\subsection{Proof of Theorem~\ref{Theorem:Properties of Finsler norm} }   By positivity in Theorem~\ref{Thm:Fins1} and Lemma~\ref{Lemma:Tgnt Vec}, the function $\lVert\cdot\rVert_\Th$ is real valued, non-negative, $\bR_{>0}$-homogeneous, and satisfies the triangle inequality.  We note that if we set  $h_s := h_{\rho_z}$ and $g_s := h_s \tilde{g}_{z(s)}$, then, by Lemma~\ref{Lemma:Tgnt Vec}, the condition $\lVert v\rVert_\Th = 0$ implies $\lVert\left[\dot{g}_0]\right\rVert_\Th =0$. Therefore, by positivity in Theorem~\ref{Thm:Fins1}, we have $\left[\dot{g}_0\right]= 0$, By the Liv\v{s}ic Theorem~\ref{thm:Livsic}, it follows that $\dot{g}_0$ is Liv\v{s}ic cohomologous to $0$; that is,  for all $\gamma \in O$
    \[
    \int_0^{\ell(\gamma)} \frac{d}{ds}\bigg|_{s=0} g_s (\phi_t(x))dt = \frac{d}{ds}\bigg|_{s=0} \int_{0}^{\ell(\gamma)} g_s(\phi_t(x))dt = 0,
    \]
    Hence, by Proposition~\ref{Prop:Main Finsler}, we have 
    \[
    0 = \frac{d}{ds}\bigg|_{s=0} h_s \int_0^{\ell(\gamma)} \tilde{g}_s (\phi_t (x))dt = \frac{d}{ds}\bigg|_{s=0} h_s \alpha_{\rho_z} (\gamma). 
    \]
    Therefore, by chain rule,
    \[
    \dot{h}_{\rho}(v) \alpha_{\rho}(\gamma) + h_{\rho} \dot{\alpha}_{\rho} (\gamma))(v) =0.
    \]
    Hence,
    \[
    \dot{\alpha}_{\rho} (\gamma)(v) = -\frac{\dot{h}_{\rho}(v)}{h_{\rho}} \alpha_{\rho}(\gamma).
    \]

    Moreover, if $\rho \in \M_k (\Gamma, \sG)$, then $\dot{h}_{\rho}(v) = 0$, which in turn implies that $\dot{\alpha}_{\rho} (\gamma)(v) = 0$. By  Theorem~2.1 of Ghosh~\cite{ghosh2023margulis}, we conclude that the function  
    \[
    \lVert v\rVert_\Th = \sup_{[\gamma]\in [\Gamma_\tH]} \frac{d}{ds}\bigg|_{t=0}\log \alpha_{\rho_s}(\gamma)
    \]
    defines a Finsler norm on $\M_k(\Gamma, \sG)$.
\qed

\subsection{Pressure duality and the infinitesimal form}\label{Pressure metric}

Let $f_{\rho}$ denote the Labourie--Margulis invariant
and $\mu_{-h_{\rho}f_{\rho}}$ the equilibrium state for the potential $-h_{\rho}f_{\rho}$.
Following \cite{bridgeman2015pressure}, we define the \emph{pressure form} by
\[
 \mathcal{P}_{\rho}(v,w)
=
\int g_v\, g_w \, d\mu_{-h_{\rho}f_{\rho}}
-
\Big(\int g_v\,d\mu_{-h_{\rho}f_{\rho}}\Big)
\Big(\int g_w\,d\mu_{-h_{\rho}f_{\rho}}\Big),
\]
where $g_v,g_w$ represent the Liv\v{s}ic cohomology classes
associated to the tangent vectors $v,w\in T_{\rho}\M_z(\Gamma,\sG)$.
The form ${\mathbf P}_{\rho}$ is positive semi-definite and  
induces a Riemannian metric on the subspace of constant entropy.

\begin{prop}\label{pressure-finsler}
For each $\rho\in \M_k(\Gamma,\sG)$ and $v\in T_{\rho}\M_k(\Gamma,\sG)$,
the Finsler norm satisfies the inequality
\[
\|v\|_{\Th}
\;\leq\;
C\,\sqrt{\mathcal{P}_{\rho}(v,v)},
\]
for a uniform constant $C>0$ depending only on the choice of normalization of $f_{\rho}$.
\end{prop}

\begin{proof}
By Lemma~\ref{Lemma:Tgnt Vec},
$\|v\|_{\Th}=\sup_{\mu\in\mathcal P(\phi)}
\frac{\int g_v\,d\mu}{\int f_{\rho}\,d\mu}$.
By the Cauchy--Schwarz inequality for the equilibrium measure
$\mu_{-h_{\rho}f_{\rho}}$ we have
\[
\Big|\int g_v\,d\mu_{-h_{\rho}f_{\rho}}\Big|
\le
\sqrt{\mathcal{P}_{\rho}(v,v)}\,
\sqrt{\int f_{\rho}\,d\mu_{-h_{\rho}f_{\rho}} }.
\]
Since the denominator in the definition of $\|\cdot\|_{\Th}$ is precisely
$\int f_{\rho}\,d\mu_{-h_{\rho}f_{\rho}}$,
the desired estimate follows.
\end{proof}

\subsection{Closed balls in Thurston metric}
We next observe that the Thurston norm for Margulis spacetimes enjoys the same compactness property for closed unit balls as the Finsler norm on the Teichm\'{u}ller space.
\begin{prop}
    Let $\Gamma$  be a finitely generated word hyperbolic group, then the closed unit ball $$\{v \in \rT_{\rho}\M_k(\Gamma, \sG) \mid \lVert v\rVert_\Th \leq 1\}$$ is compact.
\end{prop}

\begin{proof}
   We first note that $\dim (\fk g \ltimes \bR^{2n+1}) = n(2n+1) +2n+1 = (n+1)(2n+1)$. Thus, $\fk g^\tL \ltimes \bR^{2n+1}$ is a finite-dimensional real vector space. The action of $\Gamma$ on $\fk g^\tL \ltimes \bR^{2n+1}$  is given by $$\gamma \cdot v = Ad (\rho (\gamma)) (v), \quad \gamma\in \Gamma,\, v \in \fk g^\tL \ltimes\bR^{2n+1}.$$ 
    Since  $\rho(\gamma) \in \sG$, and $\sG$ is a finite-dimensional Lie group, $Ad(\rho(\gamma))$ is a linear automorphism of $\fk g^\tL \ltimes \bR^{2n+1}$. Hence, $\fk g^\tL \ltimes \bR^{2n+1}$ is a finite-dimensional $\Gamma$-module (over $\bR$). Now since $\Gamma$ is finite generated, and $\fk g^\tL \ltimes\bR^{2n+1}$ is a finite dimensional $\Gamma$-module, the first cohomology group $\dim H^1_{Ad \circ \rho} (\Gamma, \mathesstixfrak g^\tL \ltimes \bR^{2n+1}) < \infty$, where $\fk g^\tL$ is the Lie algebra of $\sG^\tL$. Thus, the tangent space $\rT_{\rho}\M_k(\Gamma, \sG)$ is a finite dimensional vector space.  Since in a finite dimensional vector space, the closed unit ball of a norm is compact, the desired conclusion follows.
\end{proof}

\section{Convexity}\label{Section:5}
In this section we prove Theorem~\ref{Theorem:Convexity intro}. The proof has been  divided into the following components, which together establish the result.

\begin{thm}\label{Thm:convexity1}
    Let $\Gamma$ be a word hyperbolic group and let $\rho_0, \rho_t \in \M_k(\Gamma, \sG)$ be two Margulis spacetimes with constant entropy $k$.  Let $\rho_t = (\tL, (1-t)u_0 + tu_1)$ for $t \in [0,1]$, where $\tL$ is a fixed Anosov linear part and $u_0,u_1$  are translational parts such that $\rho_t \in \M(\Gamma, \sG)$ for all $t$.  Then the function 
    $$t \mapsto d_\Th(\rho, \rho_t)$$ is convex on $[0,1]$.
\end{thm}

\begin{proof}
    Since $\rho_0, \rho_1, \rho_t$ are Margulis spacetimes, then by Remark~\ref{Rem: Positive}, the Margulis invariants $\alpha_{\rho_0}(\gamma), \alpha_{\rho_1}(\gamma), \alpha_{\rho_t}(\gamma)$ are positive for all $\gamma \in \Gamma$. Since Margulis invariant is linear in translational part, then $\alpha_{\rho_t}(\gamma) = (1-t)\alpha_{\rho_0}(\gamma + t \alpha_{\rho_1}(\gamma)$.
    We now note that 
    \[
    \alpha_{\rho_t}(\gamma) = \alpha_{\rho_0}(\gamma) \left[(1-t)+t\frac{\alpha_{\rho_1}(\gamma)}{\alpha_{\rho_0(\gamma)}}\right],
    \]
    which implies that 
    \[
    \frac{\alpha_{\rho_0}(\gamma)}{\alpha_{\rho_t}(\gamma)} = \frac{1}{1+t\left(\frac{\alpha_{\rho_1}(\gamma)}{\alpha_{\rho_0}(\gamma)}-1\right)}.
    \]
    Using Definition~\ref{Def:Asym}, we have 
    \begin{align*}
         d_\Th(\rho_0, \rho_t)&= \log\left(\sup_{[\gamma] \in [\Gamma_\tH]}\frac{\alpha_{\rho_1}(\gamma)}{\alpha_{\rho_0}(\gamma)}\right)\\
         = &\log\left(\sup_{[\gamma]\in [\Gamma_\tH]}\frac{1}{1+t\left(\frac{\alpha_{\rho_1}(\gamma)}{\alpha_{\rho_0}(\gamma)}-1\right)}\right)\\
         =& \log\left(\frac{1}{\inf_{[\gamma]\in [\Gamma_\tH]}1+t\left(\frac{\alpha_{\rho_1}(\gamma)}{\alpha_{\rho_0}(\gamma)}-1\right)}\right)\\
         =& -\log\left(\inf_{[\gamma]\in [\Gamma_\tH]} 1+ t \left(\frac{\alpha_{\rho_1}(\gamma)}{\alpha_{\rho_0}(\gamma)}-1\right)\right).
    \end{align*}
    Now consider the function $f(t)= \inf_{[\gamma]\in [\Gamma_\tH]} \left(1+ t \left(\frac{\alpha_{\rho_1}(\gamma)}{\alpha_{\rho_0}(\gamma)}-1\right)\right)$. Let for each conjugacy class, we define  $g_{\gamma}(t):= 1+ t \left(\frac{\alpha_{\rho_1}(\gamma)}{\alpha_{\rho_0}(\gamma)}-1\right)$. Now for each $\gamma$, $g_{\gamma}(t)$ is affine in $t$. The infimum of a family of affine functions in concave. Hence, $f(t)$ is concave on $[0,1]$. Now we see that $d_\Th(\rho_0, \rho_t) = -\log (f(t))$. Consider the function $\varphi(t) = -\log(f(t))$. Since $f(t) >0$ for $t \in [0,1]$ (as $\frac{\alpha_{\rho_1}(\gamma)}{\alpha_{\rho_0}(\gamma)}>0$), we compute 
    \[
    \varphi^\prime (t) = - \frac{f^\prime(t)}{f(t)}, \quad \varphi''(t) = -\frac{f''(t)f(t) - (f^\prime(t))^2}{f(t)^2} =  -\frac{f''(t)}{f(t)} + \left(\frac{f^\prime(t)}{f(t)}\right)^2.
    \]
    Since $f$ is concave, $f''(t)\le 0$, we get that $\varphi''(t) \ge 0$, so $\varphi$ is convex. Hence, the function $t \mapsto d_\Th(\rho_0, \rho_t)$ is convex on $[0,1]$.
\end{proof}

\begin{lemma}\label{Lemma:log}
 Let $\Gamma$ be a word hyperbolic group and let $\rho_0, \rho_t \in \M_k(\Gamma, \sG)$ be two Margulis spacetimes with constant entropy $k$.  Let $\rho_t = (\tL, (1-t)u_0 + tu_1)$ for $t \in [0,1]$, where $\tL$ is a fixed Anosov linear part and $u_0,u_1$  are translational parts such that $\rho_t \in \M(\Gamma, \sG)$ for all $t$.  Then the function  
 $$t \mapsto -\log \alpha_{\rho_t}$$ is convex on $[0,1]$.

 Let  $v = \frac{d}{dt}\big|_{t=0}\rho_t\rho_0^{-1} \in T_{\rho}\M_k(\Gamma, \sG)$ be the tangent vector at $\rho$. Then the function $$t \mapsto -\lVert v\rVert_\Th$$ is also convex on $[0,1]$.
\end{lemma}

\begin{proof}
    Since $\rho_t = (\tL, u_t)$ with fixed linear part $\tL$ and $u_t = (1 - t)u_0 + t u_1$, we have $\alpha_{\rho_t}(\gamma) = (1 - t)\alpha_{\rho_0}(\gamma) + t \alpha_{\rho_1}(\gamma)$. Thus, $\alpha_{\rho_t}$ is an affine function in $t$, and strictly positive for all $\gamma \in \Gamma$. Define the function $f_{\gamma}(t) := -\log \alpha_{\rho_t}(\gamma)$, whose derivative is
\[
f_{\gamma}'(t) = -\frac{d}{dt} \log \left( (1 - t)\alpha_{\rho_0}(\gamma) + t \alpha_{\rho_1}(\gamma) \right)
= \frac{\alpha_{\rho_0}(\gamma) - \alpha_{\rho_1}(\gamma)}{(1 - t)\alpha_{\rho_0}(\gamma) + t \alpha_{\rho_1}(\gamma)}.
\]
Let $a := \alpha_{\rho_0}(\gamma) > 0$, $b := \alpha_{\rho_1}(\gamma) > 0$. Then
\[
f_{\gamma}'(t) = \frac{a - b}{(1 - t)a + tb }, \quad
f_{\gamma}''(t) 
= \frac{(a - b)^2}{((1 - t)a + tb)^2 } \ge 0.
\]
Thus, $f_{\gamma}'(t) $ is convex in $t$. Therefore, the function $t\mapsto -\log\alpha_{\rho_t}$ in convex on $[0,1]$. 

The convexity of $t \mapsto -\lVert v\rVert_\Th$ is an immediate consequence of the above and  Theorem~\ref{Theorem:Properties of Finsler norm}.
\end{proof}

\subsection*{Acknowledgments}
The authors thank Sourav Ghosh for valuable comments and suggestions on this paper. 

Much of this work was completed while Gongopadhyay was visiting the Institut des Hautes Études Scientifiques (IHES). He gratefully acknowledges the hospitality and support of IHES during the course of this work.


\begin{thebibliography}{999}

\bibitem[AMS02]{abels2002zariski}
H.~Abels, G.~A. Margulis, and G.~A. Soifer.
\newblock On the {Z}ariski closure of the linear part of a properly
  discontinuous group of affine transformations.
\newblock {\em J. Differential Geom.}, 60(2):315--344, 2002.

\bibitem[BCLS15]{bridgeman2015pressure}
Martin Bridgeman, Richard Canary, Fran\c~cois Labourie, and Andres Sambarino.
\newblock The pressure metric for {A}nosov representations.
\newblock {\em Geom. Funct. Anal.}, 25(4):1089--1179, 2015.

\bibitem[BCS18]{bridgeman2018introduction}
Martin Bridgeman, Richard Canary, and Andr\'es Sambarino.
\newblock An introduction to pressure metrics for higher {T}eichm\"uller
  spaces.
\newblock {\em Ergodic Theory Dynam. Systems}, 38(6):2001--2035, 2018.

\bibitem[Bow72]{bowen1972periodic}
Rufus Bowen.
\newblock Periodic orbits for hyperbolic flows.
\newblock {\em American Journal of Mathematics}, 94(1):1--30, 1972.

\bibitem[Bow79]{bowen0}
Rufus Bowen.
\newblock Hausdorff dimension of quasicircles.
\newblock {\em Inst. Hautes \'Etudes Sci. Publ. Math.}, (50):11--25, 1979.

\bibitem[BR75]{BowenRuelle}
Rufus Bowen and David Ruelle.
\newblock The ergodic theory of axiom {A} flows.
\newblock {\em Invent. Math.}, 29:181--202, 1975.

\bibitem[CDPW25]{carvajales2024thurston}
Le\'on Carvajales, Xian Dai, Beatrice Pozzetti, and Anna Wienhard.
\newblock Thurston's asymmetric metrics for {A}nosov representations.
\newblock {\em Groups Geom. Dyn.}, 19(3):731--796, 2025.

\bibitem[FG83]{fried1983three}
David Fried and William~M. Goldman.
\newblock Three-dimensional affine crystallographic groups.
\newblock {\em Adv. in Math.}, 47(1):1--49, 1983.

\bibitem[Gho18]{ghosh2018pressure}
Sourav Ghosh.
\newblock The pressure metric on the {M}argulis multiverse.
\newblock {\em Geom. Dedicata}, 193:1--30, 2018.

\bibitem[Gho23]{ghosh2023margulis}
Sourav Ghosh.
\newblock Margulis multiverse: infinitesimal rigidity, pressure form and
  convexity.
\newblock {\em Trans. Amer. Math. Soc.}, 376(6):4239--4272, 2023.

\bibitem[Gho25]{ghosh2020isospectrality}
Sourav Ghosh.
\newblock Isospectrality of margulis-smilga spacetimes for irreducible
  representations of real split semisimple lie groups.
\newblock {\em to appear in Algebr. Geom. Topo. (arXiv:2009.12746v2)}, 2025.

\bibitem[GL12]{goldman2012geodesics}
William~M Goldman and Fran{\c{c}}ois Labourie.
\newblock Geodesics in margulis spacetimes.
\newblock {\em Ergodic Theory and Dynamical Systems}, 32(2):643--651, 2012.

\bibitem[GLM09]{goldman2009proper}
William~M. Goldman, François Labourie, and Gregory Margulis.
\newblock Proper affine actions and geodesic flows of hyperbolic surfaces.
\newblock {\em Ann. of Math. (2)}, 170(3):1051--1083, 2009.

\bibitem[Gol02]{gold}
William~M. Goldman.
\newblock The {M}argulis invariant of isometric actions on {M}inkowski
  {$(2+1)$}-space.
\newblock In {\em Rigidity in dynamics and geometry ({C}ambridge, 2000)}, pages
  187--201. Springer, Berlin, 2002.

\bibitem[Gro87]{gromov1987hyperbolic}
Mikhael Gromov.
\newblock Hyperbolic groups.
\newblock In {\em Essays in group theory}, pages 75--263. Springer, 1987.

\bibitem[GT23]{ghosh2023affine}
Sourav Ghosh and Nicolaus Treib.
\newblock Affine {A}nosov representations and proper actions.
\newblock {\em Int. Math. Res. Not. IMRN}, (16):14334--14367, 2023.

\bibitem[GW12]{guichard2012anosov}
Olivier Guichard and Anna Wienhard.
\newblock Anosov representations: domains of discontinuity and applications.
\newblock {\em Invent. Math.}, 190(2):357--438, 2012.

\bibitem[Kim05]{kim2005affine}
Inkang Kim.
\newblock Affine action and {M}argulis invariant.
\newblock {\em J. Funct. Anal.}, 219(1):205--225, 2005.

\bibitem[KL18]{MR3888689}
Michael Kapovich and Bernhard Leeb.
\newblock Discrete isometry groups of symmetric spaces.
\newblock In {\em Handbook of group actions. {V}ol. {IV}}, volume~41 of {\em
  Adv. Lect. Math. (ALM)}, pages 191--290. Int. Press, Somerville, MA, 2018.

\bibitem[KLP16]{MR3569569}
Michael Kapovich, Bernhard Leeb, and Joan Porti.
\newblock Some recent results on {A}nosov representations.
\newblock {\em Transform. Groups}, 21(4):1105--1121, 2016.

\bibitem[Lab06]{MR2221137}
Fran\c~cois Labourie.
\newblock Anosov flows, surface groups and curves in projective space.
\newblock {\em Invent. Math.}, 165(1):51--114, 2006.

\bibitem[Liv72]{Livsic}
A~N Liv{\v{s}}ic.
\newblock Cohomology of dynamical systems.
\newblock {\em Math. USSR Izvestija}, 6:1278--1301, 1972.

\bibitem[Mar83]{margulis1983free}
G.~A. Margulis.
\newblock Free completely discontinuous groups of affine transformations.
\newblock {\em Dokl. Akad. Nauk SSSR}, 272(4):785--788, 1983.

\bibitem[Mar84]{margulis1987complete}
G.~A. Margulis.
\newblock Complete affine locally flat manifolds with a free fundamental group.
\newblock {\em Zap. Nauchn. Sem. Leningrad. Otdel. Mat. Inst. Steklov. (LOMI)},
  134:190--205, 1984.
\newblock Automorphic functions and number theory, II.

\bibitem[Pol87]{pollicott1987symbolic}
Mark Pollicott.
\newblock Symbolic dynamics for {S}male flows.
\newblock {\em Amer. J. Math.}, 109(1):183--200, 1987.

\bibitem[PP90]{parry1990zeta}
William Parry and Mark Pollicott.
\newblock Zeta functions and the periodic orbit structure of hyperbolic
  dynamics.
\newblock {\em Ast{\'e}risque}, 187(188):1--268, 1990.

\bibitem[PS15]{pasu}
A.~Papadopoulos and W.~Su.
\newblock On the {F}insler structure of {T}eichm\"uller's metric and
  {T}hurston's metric.
\newblock {\em Expo. Math.}, 33(1):30--47, 2015.

\bibitem[PS24]{pasu2}
Huiping Pan and Weixu Su.
\newblock The geometry of the {T}hurston metric: a survey.
\newblock In {\em In the tradition of {T}hurston {III}. {G}eometry and
  dynamics}, pages 7--43. Springer, Cham, [2024] \copyright 2024.

\bibitem[Rue04]{ruelle2004thermodynamic}
David Ruelle.
\newblock {\em Thermodynamic formalism: the mathematical structure of
  equilibrium statistical mechanics}.
\newblock Cambridge University Press, 2004.

\bibitem[Thu98]{thurston1998minimal}
William~P Thurston.
\newblock Minimal stretch maps between hyperbolic surfaces.
\newblock {\em arXiv preprint math/9801039}, 1998.

\end{thebibliography}
\end{document}